\documentclass{amsart}

\usepackage{srcltx}
\usepackage{amsmath}
\usepackage{amssymb}
\usepackage{graphicx}
\usepackage{amsmath,amscd}
\usepackage{tikz}
\usepackage{tikz-cd}
\usetikzlibrary{positioning,arrows}
\usepackage[utf8]{inputenc}
\usepackage[T1]{fontenc}
\usepackage{hyperref}
\usepackage{enumitem}

\numberwithin{equation}{section}
\theoremstyle{plain}

\newtheorem{lem}[equation]{Lemma}

\newtheorem{prop}[equation]{Proposition}
\newtheorem{thm}[equation]{Theorem}
\newtheorem{cor}[equation]{Corollary}
\newtheorem{conj}[equation]{Conjecture}

\theoremstyle{definition}
\newtheorem{definition}[equation]{Definition}
\newtheorem{remark}[equation]{Remark}
\newtheorem{example}[equation]{Example}

\newcommand{\field}[1]{\mathbb{#1}}

\newcommand{\naturals}{\ensuremath{\field{N}}}

\newcommand{\C}{\mathbb C}
\newcommand{\G}{\mathcal G}

\newcommand{\R}{\mathcal R}

\newcommand{\Aut}[1]{\text{Aut}(#1)}

\newcommand{\acts}{\curvearrowright}

\newcommand{\grade}{\operatorname{grade}}
\newcommand{\proj}{\operatorname{Proj}}

\let\diameter\relax \DeclareMathOperator{\diameter}{\text{diameter}}

\newcommand{\systole}[1]{\ensuremath{\|{#1}\|}}
\newcommand{\nclose}[1]{\ensuremath{\langle\!\langle#1\rangle\!\rangle}}
\newcommand{\noobla}[2]{\ensuremath{\vert #1\vert_{#2}}}

\DeclareMathOperator{\intersector}{Intersector}

\newcommand{\subpres}{\circledast}

\newcommand{\showcomments}{yes}

\newsavebox{\commentbox}
{\ifthenelse{\equal{\showcomments}{yes}}%
{\footnotemark
        \begin{lrbox}{\commentbox}
        \begin{minipage}[t]{1.25in}\raggedright\sffamily\tiny
        \footnotemark[\arabic{footnote}]}
{\begin{lrbox}{\commentbox}}}%
{\ifthenelse{\equal{\showcomments}{yes}}%
{\end{minipage}\end{lrbox}\marginpar{\usebox{\commentbox}}}
{\end{lrbox}}}

\usetikzlibrary{calc}
\newcommand*{\vertchar}[2][0pt]{%
  \tikz[
    inner sep=0pt,
    shorten >=-.15ex,
    shorten <=-.15ex,
    line cap=round,
    baseline=(c.base),
  ]\draw [line width=0.05mm,  black]
    (0,0) node (c) {#2}
    ($(c.south)+(#1,0)$) -- ($(c.north)+(#1,0)$);%
}

\newcommand{\pstab}{\operatorname{\textup{\vertchar{S}tab}}}
\newcommand{\stab}{\operatorname{Stab}}

\title{Virtual Specialness of certain graphs of special cube complexes}

\author{Jingyin Huang}
\address{Dept. of Math. \& Stats.\\
	McGill University \\
	Montreal, Quebec, Canada}
\curraddr{Department of Mathematics, The Ohio State University, 231 W. 18th Ave, Columbus, Ohio, U.S. 43210}
\email{huang.929@osu.edu}

\author{Daniel T. Wise}
\address{Dept. of Math. \& Stats.\\
	McGill University \\
	Montreal, Quebec, Canada}
\email{daniel.wise@mcgill.ca}

\begin{document}
\maketitle

\begin{abstract}
We investigate the virtual specialness of a compact cube complex $X$ that splits as a graph
of nonpositively curved cube complexes.
We prove virtual specialness of $X$ when each vertex space of $X$ has word-hyperbolic $\pi_1$ and $\pi_1X$ has ``finite stature'' relative to its edge groups. 
The results generalize the motivating case when tree $\times$ tree lattices are virtual products.
\end{abstract}

	\tableofcontents
\section{Introduction}	

Recall the following notion which was introduced in \cite{GMRS98}:
\begin{definition}[Height]
	\label{def:height}
	A subgroup $H\le G$ has \emph{finite height} if there do not exist
	arbitrarily long sequences $\{g_1,\ldots, g_h\}$ such that $g_iH\neq g_j H$ for $i\neq j$ and 
	such that
	$\cap_{n=1}^h H^{g_n}$ is infinite.
	We use the notation $H^g=gHg^{-1}$.
	The \emph{height} of $H$  is the supremal length of such sequences.
	
	More generally, a finite collection $\mathcal H = \{H_1,\ldots, H_r\}$ of subgroups has \emph{finite height} if each $H_k$ has finite height.
\end{definition}

\begin{definition}
	\label{def:finite stature}
	Let $G$ be a group and let $\Omega = \{G_\lambda\}_{\lambda\in\Lambda}$ be a collection of subgroups of $G$. Then $(G,\Omega)$ has \emph{finite stature} if for each $H\in \Omega$, there are finitely many $H$-conjugacy classes of infinite subgroups of form $H\cap C$, where $C$ is an intersection of (possibly infinitely many) $G$-conjugates of elements of $\Omega$.
\end{definition}

For example, if $G$ is abelian, then $(G,\Omega)$ has finite stature for any finite collection $\Omega$ of subgroups of $G$. However, $\Omega$ has infinite height if some element of $\Omega$ is infinite and of infinite index. We refer Section~\ref{subsec:FGH} for more examples with finite stature but infinite height.
The notion of finite stature was introduced in
 \cite{HuangWiseSeparability} to which we refer the reader for additional examples and viewpoints. An alternative, but equivalent definition of finite stature is given in Definition~\ref{def:finite stature1}. 

\begin{definition}
	\label{def:graph of cube complexes}
	A cube complex $X$ splits as a 
	\emph{graph of nonpositively curved cube complexes}
	if it is built as follows:
There is an \emph{underlying graph} $\Gamma$, and for each vertex $v$ and edge $e$ of $\Gamma$, we associate a \emph{vertex space} $X_v$ and an \emph{edge space} $X_e$, which are nonpositively curved cube complex. $X$ is obtained from the disjoint union of the vertex spaces $X_v$ and the \emph{thickened} edge spaces $X_e\times[-1,1]$ by gluing each subcomplex $X_e\times\{-1\}$ to  $X_{\iota(e)}$ through the \emph{attaching map} $X_e\rightarrow X_{\iota(e)}$, and similarly each subcomplex $X_e\times\{1\}$ is attached to  $X_{\tau(e)}$ through $X_e\rightarrow X_{\tau(e)}$. The attaching maps are local-isometries.
\end{definition}

We review the notion of special cube complex in Section~\ref{subsec:cube complex}.
 A cube complex is \emph{virtually special} if it has a finite degree cover that is special.

\begin{thm}
	\label{thm:main}
	Let $X$ split as a graph of nonpositively curved cube complexes.
	Suppose $X$ is compact and the fundamental group of each vertex space is word-hyperbolic.
	Then the following are equivalent.
	\begin{enumerate}
		\item $\pi_1 X$ has finite stature with respect to its vertex groups.
		\item $X$ is virtually special.
	\end{enumerate}
\end{thm}
We conjecture that Theorem~\ref{thm:main} holds without assuming the vertex groups are word-hyperbolic. Actually, the direction $(2)\Rightarrow(1)$ follows from Theorem~\ref{thm:FGH}, which does not rely on this hyperbolicity assumption.

\begin{conj}
	\label{conj:main}
	Let $X$ split as a graph of nonpositively curved cube complexes.
	Suppose $X$  is compact and  each vertex space is virtually special. Then the following are equivalent.
	\begin{enumerate}
		\item $\pi_1 X$ has finite stature with respect its vertex groups.
		\item $X$ is virtually special.
	\end{enumerate}
\end{conj}

In the long term, we hope the following more complete characterization of virtual specialness will become approachable. The ``only if'' direction of this conjecture holds by Theorem~\ref{thm:FGH}. 
\begin{conj}
	Let $X$ be a compact non-positively curved cube complex. Then $X$ is virtually special if and only if $\pi_1 X$ has finite stature with respect to the collection of its hyperplane stabilizers.
\end{conj}

Special cases of Conjecture~\ref{conj:main} when the vertex groups are not hyperbolic are treated in \cite{LiuVirtual,PrzytyckiWiseGraphManifolds,PrzytyckiWiseMixed,HuangComensurability}. Several outstanding open questions for lattices acting on cube complexes, as well as quasi-isometric classification and rigidity of certain cubical groups, can be reduced to Conjecture~\ref{conj:main} \cite{Haglund2006,HuangKleinerGroups}.

Though we require the vertex group to be word-hyperbolic in Theorem~\ref{thm:main}, $X$ does not have to be hyperbolic or relative hyperbolic. The key example to have in mind is where $\pi_1X$ is a lattice acting on a product of two trees, and $X$ has the structure of a graph of graphs. 
In that case, $\pi_1 X$ is reducible if and only if $\pi_1 X$ has finite stature with respect to its vertex groups. In particular, irreducible lattices acting on products of trees \cite{Wise96Thesis,BurgerMozes2000} do not have finite stature with respect to their vertex groups.

The notion of finite stature naturally arises when one examines the pathological behavior of these irreducible lattices. 
Several equivalent conditions of finite stature are discussed in Section~\ref{sec:height big-trees and depth}.

It is an open problem, going back to the origin of special cube complexes \cite{HaglundWiseSpecial}, whether
being virtually special is implied by (and thus equivalent to) having separable hyperplane subgroups. This has been a continuing theme in the topic. It was first proven in the $\mathcal{V}\mathcal{H}$ case \cite{WiseFigure8}, it was proven in the hyperbolic case in \cite{HaglundWiseAmalgams},
and in the relatively hyperbolic case in \cite{WiseIsraelHierarchy}.
There has been little progress in general without assuming some version of hyperbolicity
and we provide a version with a fairly muted hyperbolicity assumption:
\begin{cor} (cf.\ Corollary~\ref{cor:separable hyperplanes})
Let $X$ split as a graph of nonpositively curved cube complexes,
with $X$ compact and each $\pi_1X_v$  is word-hyperbolic. 
If the stabilizer of each hyperplane in $\widetilde X$ is separable
then $X$ is virtually special.
\end{cor}

The following is a consequence of ideas used in the proof of Theorem~\ref{thm:main}, and is of independent interest. It is related to a key technical point in \cite{HaglundWiseAmalgams}.
\begin{thm} (cf.\ Corollary~\ref{cor:connected intersection})
	Let $X$ be a compact special cube complex such that $\pi_1 X$ is word-hyperbolic. Let $\{f_i:Y_i\to X\}_{i=1}^n$ be a collection of local-isometries with compact domains. Then there exists a finite regular cover $\widehat X$ of $X$ such that for any $\{\widehat Y_j\to \widehat X\}_{j=1}^3$ where each $\widehat Y_j$ is an elevation of an element in $\{Y_i\}$, the following holds:
	\begin{enumerate}
		\item Each $\widehat Y_j$ embeds.
		\item If $\widehat Y_1\cap \widehat Y_2\neq\emptyset$, then $\widehat Y_1\cap \widehat Y_2$ is connected.
		\item If $\{\widehat Y_1,\widehat Y_2,\widehat Y_3\}$ pairwise intersect, then $\widehat Y_1\cap\widehat Y_2\cap\widehat Y_3\neq\emptyset$.
	\end{enumerate}
	Moreover, for a given collection of finite covers $\{Y'_i\to Y_i\}_{i=1}^n$, we can require that $\widehat Y_i\to Y_i$ factors through $Y'_i\to Y_i$ for each $i$.
\end{thm}


\section{Preliminary}

\subsection{Background on Cube Complexes}\label{sub:cubical background}
\label{subsec:cube complex}
\subsubsection{Nonpositively curved cube complexes:}
An \emph{$n$-dimensional cube} is a copy of $[-\frac12,+\frac12]^n$.
Its \emph{subcubes} are the subspaces obtained by restricting some coordinates to $\pm\frac12$.
We regard a subcube as a copy of a cube in the obvious fashion.
A \emph{cube complex} $X$ is a cell complex obtained by gluing cubes together along subcubes,
where all gluing maps are modeled on isometries.
Recall that a \emph{flag complex} is a simplicial complex with the property that a finite set of vertices spans a simplex if and only if they are pairwise adjacent.
$X$ is \emph{nonpositively curved} if the link of each $0$-cube of $X$ is a flag complex.
A \emph{CAT(0) cube complex} $\widetilde X$ is a simply-connected nonpositively curved cube complex.

\subsubsection{Hyperplanes}
A \emph{midcube} is a subspace of an $n$-cube obtained by restricting one coordinate of $[-\frac12,+\frac12]^n$ to $0$.  A \emph{hyperplane} $\widetilde U$ is a  connected subspace of a CAT(0) cube complex $\widetilde X$
such that for each cube $c$ of $\widetilde X$, either  $\widetilde U\cap c =\emptyset$ or $\widetilde U\cap c$ consists of a midcube of $c$. The \emph{carrier} of a hyperplane $U$ is the subcomplex $N(\widetilde U)$ consisting of all closed cubes intersecting $U$.
We note that every midcube of $\widetilde X$ lies in a unique hyperplane, and $N(\widetilde U)\cong \widetilde U\times c^1$ where $c^1$ is a 1-cube. An \emph{immersed hyperplane} $U\rightarrow X$ in a nonpositively curved cube complex is a map $\stab(\widetilde U)\backslash \widetilde U \rightarrow X$ where $\widetilde U$ is a hyperplane of the universal cover $\widetilde X$ of $X$. We similarly define $N(U)\rightarrow X$ via $N(U)=\stab(\widetilde U)\backslash N(\widetilde U)$.

\newcommand{\link}{\text{link}}

A map $\phi:Y\rightarrow X$ between nonpositively curved cube complexes is \emph{combinatorial} if it maps open $n$-cubes homeomorphically to open $n$-cubes. A combinatorial map is a \emph{local-isometry}
if for each $0$-cube $y$,
the induced map $\link(y)\rightarrow \link(\phi(y))$ is an embedding of simplicial complexes,
such that $\link(y)\subset \link(\phi(y))$ is \emph{full} in the sense that if a collection of vertices
of $\link(y)$ span a simplex in $\link(\phi(y))$ then they span a simplex in $\link(y)$.

\subsubsection{Special Cube Complexes}
A nonpositively curved cube complex $X$ is \emph{special}
if each immersed hyperplane $U\rightarrow X$  is an embedding, and moreover $N(U)\cong U\times [-\frac12,+\frac12]$,
each restriction $U\times \{\pm\frac12\} \rightarrow X$ is an embedding,
and if $U,V$ are hyperplanes of $X$ that intersect then a 0-cube of $N(U)\cap N(V)$ lies in a 2-cube
intersected by both $U$ and $V$.

\subsection{Cubical small cancellation}
A \emph{cubical presentation} $\langle X | \{Y_i\} \rangle$ consists of a nonpositively curved cube complex $X$, and a set of local-isometries $Y_i \rightarrow X$ of nonpositively curved cube complexes. We use the notation $X^*$ for the cubical presentation above,
and each  $Y_i$ is called a \emph{cone} of $X^*$. As a topological space, $X^*$ consists of $X$ with a (genuine) cone on each $Y_i$ attached,
so $\pi_1X^*=\pi_1X/\nclose{\{\pi_1Y_i\}}$. We use the notation $\widetilde X^*$ for the universal cover of $X^*$. The complex $X$ is  the \emph{cubical part} of $X^*$.

\begin{definition}
	\label{def:morphism}
Let $Y_i\rightarrow X$ and $Y_j\rightarrow X$ be maps.
A \emph{morphism} $Y_i\rightarrow Y_j$ is a map such that $Y_i\rightarrow X$ factors as $Y_i\rightarrow Y_j\rightarrow X$. It is an \emph{isomorphism} if their is an inverse map $Y_j\rightarrow Y_i$ that is also a morphism.
We define an \emph{automorphism} accordingly and let $\Aut{Y\rightarrow X}$ denote the group of automorphisms of
$Y\rightarrow X$.
\end{definition}


A \emph{cone-piece} of $X^*$ in $Y_i$ is 
an intersection $g\widetilde Y_j \cap \widetilde Y_i$ for some $g \in \pi_1X$,
but where $g\widetilde Y_i$ is not a subcomplex of $\widetilde Y_i$ with the inclusion
 $g\widetilde Y_j\hookrightarrow \widetilde Y_i$ descending to a morphism $Y_j\rightarrow Y_i$, and $\widetilde Y_i$ is not a subcomplex of $g\widetilde Y_i$ with the inclusion
 $ \widetilde Y_i\hookrightarrow g\widetilde Y_j$ descending to a morphism $Y_i\rightarrow Y_j$.  A \emph{wall-piece} of $X^*$ in $Y_i$ is $\widetilde{Y}_i \cap N(\widetilde{U})$, where $N(\widetilde U)$ is the carrier of a hyperplane $\widetilde{U}$ 
 with $\widetilde U\cap \widetilde{Y}_i = \emptyset$. A \emph{piece} is either a cone-piece or a wall-piece.
 
\begin{remark}
The cone-piece and wall-piece defined here are called ``contiguous abstract cone piece'' and ``contiguous abstract wall-piece'' in \cite{WiseIsraelHierarchy}. There are several other type of pieces discussed in \cite{WiseIsraelHierarchy}, however, in the light of \cite[Lem~3.7]{WiseIsraelHierarchy}, these two kinds of pieces are all what we need for defining cubical small cancellation.
\end{remark}

\begin{definition}
	Let $X^*=\langle X \mid \{Y_i\}\rangle$ and $A^*=\langle A \mid \{B_j\}\rangle$ be cubical presentations.
	A \emph{map} $A^*\rightarrow X^*$ of cubical presentations is a local-isometry $A\rightarrow X$,
	so that for each $j$ there exists $i$ and  a map $B_j\rightarrow Y_i$
so that the composition $B_j\rightarrow A\rightarrow X$ equals $B_j \rightarrow Y_i\rightarrow X$.
	
	Given a cubical presentation $X^*$ and a local-isometry $A\rightarrow X$,
	 the \emph{induced presentation} is the cubical presentation of the form
	$A^* = \langle A \mid \{A\otimes_X Y_i\} \rangle$ where $A\otimes_X Y_i$ is the fiber-product of $A\rightarrow X$ and $Y_i\rightarrow X$ as in Definition~\ref{def:fiber-product}.
	Note that there is a map of cubical presentations $A^*\rightarrow X^*$.
\end{definition}

\begin{definition}[Graded Presentations and Subpresentations]\label{defn:graded pres and subpres}
A \emph{graded cubical presentation} $\langle X \mid \{Y_i\}\rangle$ is equipped with a
\emph{grading} of its cones, which is a map from the set of cones to $\naturals$.
We sometimes use the notation $\langle X \mid \{Y_{ij}\} \rangle$
to indicate that $\grade(Y_{ij})=i$.

For a cone $Y$ of $X^*$, the \emph{subpresentation} \index{$Y^\subpres$} $Y^\subpres$
is the cubical presentation induced by
 $Y\rightarrow X$ and $X_{\grade(Y)-1}^*$
where  $X_{\grade(Y)-1}^*$ is the subpresentation of $X^*$ that includes all cones whose
grade is less than $\grade(Y)$.
In many of our applications, $Y^\subpres$ is the cubical presentation whose base is $Y$
 and whose cones are either contractible, or consist of lower grade cones of $X^*$ that properly factor through $Y$.
\end{definition}


Let  $\systole{Y^\subpres}$ denote the infimal length of a closed path in $Y$
that is essential in $Y^\subpres$.

Let $\widetilde Y$ and $\widetilde Y^\subpres$ be the universal covers of $Y$ and $Y^\subpres$. There is a covering map $\xi$ from $\widetilde Y$ to the cubical part of $\widetilde Y^\subpres$. For a subcomplex $P$ of $\widetilde Y$, let $\noobla{P}{Y^\subpres}$  be the diameter of $\xi(P)$ in the cubical part of $\widetilde Y^\subpres$. For a map $P\rightarrow Y$ with trivial $\pi_1 Y^\subpres$ image, let  $\noobla{P}{Y^\subpres}$ be the diameter
of a lift of $P$ to $\widetilde Y^\subpres$.

We say $X^*$ satisfies the $C'(\frac{1}{24})$ \emph{small cancellation} condition if
$\noobla{P}{Y_i^\subpres} < \frac{1}{24} \systole{Y_i^\subpres}$
for every cone-piece or wall-piece $P$ of  $Y_i$.

The following is a slightly more restrictive version of the same notion treated in \cite[Def~3.61 and Def~3.65]{WiseIsraelHierarchy}:
\begin{definition}
	Let $A^*=\langle A\mid \{B_j\}\rangle$ and $X^*=\langle X\mid\{Y_i\}\rangle$.
	We say $A^*\rightarrow X^*$ has \emph{liftable shells} provided the following holds:
	Whenever $QS\rightarrow Y_i$ is an essential closed path with $|Q| > |S|_{Y^\subpres_i}$ and $Q\to Y_i$ factors through $Q\to A \otimes_X Y_i\to Y_i$,
	there exists $B_j$ and a lift $QS\rightarrow B_j$, such that $B_j\rightarrow A\rightarrow X$
	factors as $B_j\rightarrow Y_i\rightarrow X$.
\end{definition}
The following is a restatement of a combination of \cite[Thm~3.68 and Cor~3.72]{WiseIsraelHierarchy}:
\begin{lem}\label{lem:quasi-isom embedding}
Let  $X^*$ be $C'(\frac{1}{24})$.
Let $A^*\rightarrow X^*$ have liftable shells and suppose that $A^*$ is compact.
Then $\pi_1 A^*\to \pi_1 X^*$ is injective, and $A^*\rightarrow X^*$ lifts to an embedding $\widetilde A^*\rightarrow \widetilde X^*$ that is also a quasi-isometric embedding.
\end{lem}

The following is proven in \cite[Lem~3.67]{WiseIsraelHierarchy}:
\begin{lem}\label{lem:liftable shell criterion}
	Let $\langle X \mid \{Y_i\}\rangle$ be a $C'(\frac{1}{24})$ small-cancellation cubical presentation.
	Let $A\rightarrow X$ be a local-isometry and let $A^*$ be the associated induced presentation.
	Suppose that for each $i$ and each component $K$ of $A\otimes_X Y_i$,  either $K$ maps isomorphically to $Y_i$ or
	$\diameter(K)\leq \frac12\systole{Y_i^\subpres}$ and $\pi_1K^*=1$ where
	$K^*$ is induced by $K\rightarrow Y_i^\subpres$. Then the natural map $A^*\rightarrow X^*$ has liftable shells.
\end{lem}
Note that $K$ lifts to $\widetilde Y_i^\subpres$.
A natural scenario is when each component $K$ of $A\otimes_X
Y_i$
is either a copy of $Y_i$ or satisfies $\diameter(K)\leq \frac12\systole{Y_i^\subpres}$ with
$K$ either a contractible cube complex or a copy of a cone $Y_j$ with $\grade(Y_j)< \grade(Y_i)$.



\subsection{Helly property for cones}
\begin{definition}\label{defn:well-embedded}
	A cubical presentation $X^*$ has \emph{well-embedded cones} if the following conditions hold:
	\begin{enumerate}
		\item\label{embed:1} Let $Y_1$ be a cone of $X^*$. Then $Y_1\rightarrow X$ is injective.
		\item\label{embed:2} Let $Y_1,Y_2$ be cones in $X^*$. Then $Y_1\cap Y_2$ is connected.
		\item\label{embed:3} Let $Y_1,Y_2,Y_3$ be cones in $X^*$.
		If $Y_i\cap Y_j\neq \emptyset$ for each $i,j$ then $Y_1\cap Y_2\cap Y_3\neq \emptyset$.
	\end{enumerate}
\end{definition}

\begin{example}
	Let $X$ be a graph that is a 3-cycle with edges $a,b,c$ so that $abc$ is a path in $X$.
	Properties~\eqref{embed:1}~\eqref{embed:2},~and~\ref{embed:3}
	fail for the following cubical presentations which satisfy $C'(\alpha)$ for each $\alpha>0$:
$$\langle X \mid abc\rangle \hspace{1cm} \langle X \mid ab, c\rangle
\hspace{1cm} \langle X \mid a, b ,c \rangle$$
	Moreover, let $Z^* = \langle X \mid X, a, b,c \rangle$. Then $Z^*$ is $C'(\alpha)$
	but $\widetilde Z^* = Z^*$ fails to have well-embedded cones. Note that we then have
	$X^\subpres= \langle X \mid a, b ,c \rangle$.
\end{example}

A graded cubical presentation  $X^*$ has \emph{small subcones} if for each cone $Y$ of $X^*$,
and each cone $Z$ of $Y^\subpres$, we have $\diameter(Z)<\frac13\systole{Y^\subpres}$.

The following is proven in \cite[Lem~3.58]{WiseIsraelHierarchy}:
\begin{lem}[Well-embedded Cones]\label{lem:well-embedded cones}
	Let $X^*$ be a $C'(\frac{1}{12})$ graded metric cubical presentation with finitely many grades
	and with small subcones.
	Then $\widetilde X^*$ has well-embedded cones.
\end{lem}

\subsection{Superconvexity and fiber-products}
The following are quoted from \cite[Def~2.35 \& Lem~2.36]{WiseIsraelHierarchy}:
\newcommand{\neb}{\ensuremath{\mathcal N}}
\begin{definition}
	Let $X$ be a metric space.
	A subset $Y\subset X$ is \emph{superconvex}
	if it is convex and for any bi-infinite geodesic $\gamma$, if $\gamma$ is contained in the $r$-neighborhood $\neb_r(Y)$ for some $r>0$,
	then $\gamma\subset Y$.
	A map $Y\rightarrow X$ is \emph{superconvex}
	if the map $\widetilde Y \rightarrow \widetilde X$ is an embedding onto a superconvex subspace.
\end{definition}

\begin{lem}\label{lem:superconvex core}
	Let $H$ be a quasiconvex subgroup of a word-hyperbolic group $G$.
	And suppose that $G$ acts properly and cocompactly on
	a CAT(0) cube complex $X$.
	For each compact subcomplex $D\subset X$
	there exists a superconvex $H$-cocompact subcomplex $K\subset X$
	such that $D\subset K$.
\end{lem}

%
%

The following is a consequence of \cite[Lem~2.39]{WiseIsraelHierarchy}:
\begin{lem}
	\label{lem:bound on wall pieces}
Let $Y\rightarrow X$ be compact and superconvex.
Then there exists $r$ bounding the diameter of every wall-piece in $\langle X\mid Y\rangle$.
 \end{lem}

We record the following from \cite[Sec~8]{WiseIsraelHierarchy}:
\begin{definition}[Fiber-product]
	\label{def:fiber-product}
Given a pair of combinatorial maps $A\rightarrow X$ and $B\rightarrow X$ between cube complexes,
their \emph{fiber-product} $A\otimes_X B$ is the cube complex
whose $i$-cubes are pairs of $i$-cubes in $A,B$ that map to the same
$i$-cube in $X$.  There is a commutative diagram:
$$\begin{matrix}
A\otimes_X B & \rightarrow & B \\
\downarrow & & \downarrow \\
A & \rightarrow & X \\
\end{matrix}$$
Note that $A \otimes_X B$ is the subspace of $A\times B$ that is the preimage of the diagonal $D\subset X\times X$
under the map $A\times B\rightarrow X\times X$.
For any cube $Q$, the diagonal of  $Q\times Q$ is isomorphic to  $Q$ by either of the projections,
and this makes $D$ into a cube complex isomorphic to $X$.
Thus  $A\otimes_X B$ has an induced cube complex structure.

Our description of $A\otimes_X B$ as a subspace of $A\times B$
endows the fiber-product $A\otimes_X B$ with the property of being a universal receiver in the following sense:
Consider a commutative diagram as below.
Then there is an induced map $C\rightarrow A \otimes_X B$ such that the following diagram commutes:
$$\begin{matrix}
C  &  &\longrightarrow &     & B \\
   &\rotatebox{-45}{$\dashrightarrow$} &  & \nearrow &  \\
\downarrow &     & A\otimes_X B &  & \downarrow \\
&\swarrow && \searrow&\\
A &  & \longrightarrow & & X \\
\end{matrix}$$
\end{definition}

Let $\widetilde A$ be an elevation of $A$ to the universal cover $\widetilde X\to X$. By choosing a basepoint, we can identify $A$ with a subgroup of $\stab_{\pi_1 X}(\widetilde A)$. Define $\widetilde B$ similarly. Then a component of $A\otimes_X B$ can alternatively be described as $(g_1\widetilde A \cap g_2\widetilde B)/(g_1Ag^{-1}_1\cap g_2Bg^{-1}_2)$ for some $g_1,g_2\in\pi_1 X$.

Let $\langle X | \{Y_i\} \rangle$ be a cubical presentation. Thus any cone piece of $Y_i$ can be written as the universal cover of some component of $Y_j\otimes_Y Y_i$ in $\widetilde Y_i$.

\begin{lem}\label{lem:superconvex fiber-product intersection}
Let $A\rightarrow X$ and $B\rightarrow X$ be local-isometries of connected nonpositively curved cube complexes.
Suppose the induced embedding of universal covers $\widetilde A \hookrightarrow \widetilde X$ is  superconvex.
Then the noncontractible components of $A\otimes_X B$ correspond  to the nontrivial intersections of conjugates of $\pi_1(A,a)$ and $\pi_1(B,b)$ in $\pi_1X$.
\end{lem}

\begin{remark}\label{rem:multiple fiberproduct similar}
	Lemma~\ref{lem:superconvex fiber-product intersection} holds in general for finitely many factors.
	Specifically, the the components of $A_1\otimes_X A_2\otimes_X \cdots \otimes_X A_r$ correspond to the [nontrivial]
	intersections of conjugates of $\pi_1A_1, \pi_1A_2,\ldots, \pi_1A_r$.
\end{remark}

\subsection{Symmetrization and Principalization}
We use the following notation for conjugation: $H^g=gHg^{-1}$.
Two subgroups $H_1,H_2$ of $G$ are \emph{commensurable in $G$}, if there is $g\in G$ such that $H^g_1\cap H_2$ is of finite index in $H^g_1$ and in $H_2$. The \emph{commensurator} of $H$ in $G$, denoted by $\C_G(H)$, is the collection of $g\in G$ such that $H\cap H^g$ is finite index in both $H$ and $H^g$.
\begin{definition}
	Suppose $Y$ is connected. A local-isometry $Y\to X$ is \emph{symmetric} if for each component $K$ of $Y\otimes_X Y$, either $K$ maps isomorphically to each copy of $Y$, or $[\pi_1 Y:\pi_1 K]=\infty$.
\end{definition}

The following is \cite[Lem~8.12]{WiseIsraelHierarchy}.
\begin{lem}
	\label{lem:symmetric characterization}
Let $Y\to X$ be a superconvex local-isometry with $Y$ compact. Then $Y$ is symmetric if and only if $\C_{\pi_1 X}(\pi_1 Y)=\stab(\widetilde Y)$ and $\pi_1 Y\le \stab(\widetilde Y)$ is a normal subgroup.
\end{lem}

The following result implies that if $\pi_1 X$ is word-hyperbolic, then for a given superconvex local-isometry $Y\to X$ with $Y$ compact, we can produce a component $K$ of multiple fiber-products of $Y\to X$ such that $K$ is symmetric and $\pi_1K$ is of finite index in $\pi_1Y$.

\begin{lem}
	\label{lem:commensurator}
	\cite{KapovichShort96}
	Let $H$ be a quasiconvex subgroup of a word-hyperbolic group $G$. Then $H$ has finite index in the commensurator of $H$ inside $G$.
\end{lem}
\begin{definition}\label{defn:principal}
	Let $Z$ be a component of a multiple fiber-product of $Y=\sqcup_i Y_i$ where each $Y_i$ is superconvex. Let $\mathcal{C}$ be the collection of subgroups of $\pi_1 Y$ that are intersections of finitely many conjugates of elements in $\{\pi_1 Y_i\}$.
Then $Z$ is \emph{principal} if for any component $W$ of a multiple fiber-product of $Y=\sqcup_i Y_i$ and any component $K$ of $Z\otimes_X W$,	the map $K\rightarrow Z$ is either an isomorphism or satisfies $[\pi_1Z:\pi_1K]=\infty$. By Lemma~\ref{lem:superconvex fiber-product intersection}, Remark~\ref{rem:multiple fiberproduct similar} and Lemma~\ref{lem:symmetric characterization}, $Z$ is principal if and only if the following conditions hold simultaneously:
\begin{enumerate}
	\item $Z$ is symmetric;
	\item $\pi_1 Z$ does not contain any element of $\mathcal{C}$ as a proper finite index subgroup;
	\item for any component $W$ of a multiple fiber-product of $Y$ with $\pi_1 W=\pi_1 Z$ (up to conjugacy in $\pi_1 X$), $Z\to X$ factors through $W\to X$.
\end{enumerate}
\end{definition}

\begin{definition}
	\label{def:stable}
	Let $\mathcal{Y}=\sqcup Y_i$ be a disjoint union of compact connected components. A local-isometry $\mathcal{Y}\to X$ is \emph{stable} if
	\begin{enumerate}
		\item \label{symmetric}each $Y_i\to X$ is symmetric;
		\item \label{subgroup}if $Y_i\to X$ factors through $Y_j\to X$ via $Y_i\to Y_j\to X$, then either $Y_i\to Y_j$ is an isomorphism, or $[\pi_1 Y_j:\pi_1 Y_i]=\infty$;
		\item \label{intersection}for a component $K$ of $Y_i\otimes_X Y_j$ (it is possible that $i=j$), either at least one of the maps $K\to Y_i$, $K\to Y_j$ is an isomorphism, or there is $Y_k$ such that $Y_k\to X$ factors through $K\to X$ via $Y_k\to K\to X$ and $[\pi_1 K:\pi_1 Y_k]<\infty$.
	\end{enumerate}
\end{definition}
By Definition~\ref{def:stable}.\eqref{subgroup} and \eqref{intersection}, in the second case of (3) we can assume $[\pi_1 Y_i:\pi_1 Y_k]=\infty$ and $[\pi_1 Y_j:\pi_1 Y_k]=\infty$. Moreover, a stable $\mathcal{Y}\to X$ satisfies:
\begin{enumerate}
	\item if $\pi_1 Y_i$ has a finite index subgroup contained in $\pi_1 Y_j$ up to conjugacy in $\pi_1 X$, then $Y_i\to X$ factors through $Y_j\to X$;
	\item for a component $K$ of $Y_{i_1}\otimes_X Y_{i_2}\otimes_X \cdots \otimes_X Y_{i_n}$, there is $Y_k$ such that $Y_k\to X$ factors through $K\to X$ and $[\pi_1 K:\pi_1 Y_k]<\infty$.
\end{enumerate}

Suppose $\mathcal{Y}$ is stable. If we take the principal components of fiber-products of elements in $\mathcal{Y}$ to obtain a new collection $\mathcal{Y}'$, then elements of $\mathcal{Y}'$ are isomorphic to elements of $\mathcal{Y}$ and vice versa.

Lemma~\ref{lem:stable1} and Lemma~\ref{lem:stable2} are consequences of Definition~\ref{def:stable}. The meaning of isomorphism in the statements of these lemmas is as in Definition~\ref{def:morphism}.

\begin{lem}
	\label{lem:stable1}
Let $f:\mathcal{Y}=\sqcup Y_i\to X$ be a local-isometry and suppose there are only finitely many isomorphism classes of multiple fiber-products of $\mathcal{Y}=\sqcup Y_i$. Let $\{Z_i\}$ be the collection of representatives from isomorphism classes of principal components of multiple fiber-products of $\mathcal{Y}=\sqcup Y_i$. Then $\{Z_i\}$ is stable.
\end{lem}

\begin{lem}
	\label{lem:stable2}
	Let $\mathcal{Y}$ be a stable collection such that none of its components are
	 isomorphic. Let $\mathcal{Y}_1$ be a new collection formed by replacing each element of $\mathcal{Y}$ by a finite cover, and let $\mathcal{Y}_2$ be representatives of isomorphism classes of principal components of multiple fiber-products of elements in $\mathcal{Y}_1$. Then each element in $\mathcal{Y}_2$ is a finite cover of a unique element in $\mathcal{Y}$, and this gives a 1-1 correspondence between $\mathcal{Y}_2$ and $\mathcal{Y}$.
\end{lem}

\begin{lem}
	\label{lem:stable}
	Suppose $\mathcal{Y} \to X$ is stable where $\mathcal Y=\sqcup Y_i$ and each $Y_i$ is compact. Let $\mathcal{\widehat Y}=\sqcup \widehat Y_i$ where each $\widehat Y_i\rightarrow Y_i$ is a finite cover. Suppose
	\begin{enumerate}
		\item $\pi_1\widehat Y_i\trianglelefteq\C_{\pi_1 X}(\pi_1\widehat Y_i)$;
		\item for $i\neq j$, if $Y_i\to X$ factors through $Y_j\to X$ via $f:Y_i\to Y_j$, then $f$ factors through $\widehat Y_j\to Y_j$.
	\end{enumerate}
	Then $\mathcal{\widehat Y}\to X$ is stable.
\end{lem}

\begin{proof}
	It suffices to verify conditions in Definition~\ref{def:stable}. We only verify Definition~\ref{def:stable}.(3) since the other conditions are similar and simpler. Let $\widehat K$ be a component of $\widehat Y_i\otimes_{X}\widehat Y_j$. Let $\hat p\in\widehat K$ be a basepoint represented by $(\hat a,\hat b)$ with $\hat a\in\widehat Y_1$ and $\hat b\in\widehat Y_2$. Then $(\widehat K,\hat p)$ covers a component $(K,p)$ of $Y_i\otimes_X Y_j$ with $p$ represented by $(a,b)$. Let $Y_k\to K$ be as Definition~\ref{def:stable}.(3). Choose a point $c\in Y_k$ that is mapped to the basepoint $p\in K$. Note that $Y_k\to X$ factors through $Y_i\to X$ via the composition $f_i:Y_k\to K\to Y_i\otimes_X Y_j\to Y_i$. By assumption (2) of the lemma,
	each  $f_i$ factors through $\hat f_i:Y_k\to \widehat Y_i$. Since $\widehat Y_i\to Y_i$ is regular by (1), we can assume $f_i(c)=\hat a$. Similarly, we define $\hat f_j:Y_k\to \widehat Y_j$ such that $f_j(c)=\hat b$. By the universal property of fiber-products, $\hat f_i$ and $\hat f_j$ induces $f:Y_k\to \widehat Y_i\otimes\widehat Y_j$. Since $f(c)=(\hat a,\hat b)$, we actually have $f:Y_k\to \widehat K$, and we are done by considering the composition $\widehat Y_k\to Y_k\to \widehat K$.
\end{proof}

%

\section{Stature, big-trees and depth-reducing quotients}
\label{sec:height big-trees and depth}
\subsection{Big-trees and Stature}
\label{subsec:big-tree}
We review several notions from \cite[Sec~3.1]{HuangWiseSeparability}. Let $G$ be the fundamental group of a finite graph of groups with underlying graph $\mathcal{G}$, and let $T$ be the associated Bass-Serre tree. 
A subtree $S\subset T$ is \emph{nontrivial} if $S$ contains at least one edge.
Let $\pstab(S)$ denote the pointwise stabilizer of $S$.
When is $S$ is nontrivial, $\pstab(S)=\cap_{e \in \text{Edges}(S)} \pstab(e)$,
and so  $\pstab(S)$ equals the intersection of conjugates of edges groups of $G$. 

\begin{definition}
	\label{def:big-tree}
	\cite[Def~3.1]{HuangWiseSeparability}
A \emph{big-tree} is a subtree $S\subset T$ such that
\begin{itemize}
	\item $S$ is nontrivial;
	\item $\pstab(S)$ is infinite;
	\item there does not exist
	a subtree $S'\subset T$ with $S\subsetneq S'$ and $\pstab(S)=\pstab(S')$.
\end{itemize}
\end{definition}


Choose a maximal tree of $\mathcal{G}$ and lift this tree to a subtree  $T_{\mathcal{G}}\subset T$. This gives  an identification of vertex groups of $G$ and stabilizers of vertices in $T_{\mathcal{G}}$. For each vertex $v\in T$,  choose $g_v\in G$ with $g_v v\in T_{\mathcal{G}}$. (Though 
$g_v v$ is unique, there may be multiple choices for $g_v$.)


A \emph{based big-tree} $(S,v)$ consists of a big-tree $S\subset T$ and a vertex $v\in S$. For each $(S,v)$, define an \emph{$(S,v)$-transection} to be a subgroup of $\stab(g_vv)$ of form $g_v\pstab(S)g^{-1}_v$. Different choices of $g_v$ yield different $(S,v)$-transections, but they are  conjugate to each other within $\stab(g_vv)$. For two different vertices $v_1,v_2\in S$, the inclusions $ \stab(v_1)\hookleftarrow\pstab(S)\hookrightarrow \stab(v_2)$ induce an isomorphism between an $(S,v_1)$-transection and an $(S,v_2)$-transection. This is called a \emph{transfer isomorphism}, and is well-defined up to conjugacy in the vertex groups.

\begin{lem}
	\label{lem:correspondence}
	\cite[Lem~3.3]{HuangWiseSeparability}
Let $(S_1,v_1)$ and $(S_2,v_2)$ be based big-trees. Then there exists $g\in G$ such that $g(S_1,v_1)=(S_2,v_2)$ if and only if there exist $g_1,g_2\in G$ and $w\in T_{\mathcal{G}}$ such that $g_1v_1=g_2v_2=w$ and any $(S_1,v_1)$-transection and $(S_2,v_2)$-transection are conjugate in $\stab(w)$.
\end{lem}

\begin{definition}[$\Upsilon$ and $\Upsilon_V$]
	\label{defn:phi}\cite[Def~3.4]{HuangWiseSeparability}
Consider the action of $G$ on the collection of based big-trees. For each $G$-orbit, choose a representative $(S,v)$ and consider an $(S,v)$-transection. Let $\Upsilon$ be the collection of such transections. For a vertex group $V$ of $G$ (above we identified $V$ with the stabilizer of a vertex in $T_{\mathcal{G}}$\,), let $\Upsilon_V\subset\Upsilon$ be the subcollection of $(S,v)$-transections with  $g_vv$ corresponding to the vertex group $V$.
\end{definition}

\begin{definition}
	\label{def:lowest and higher} \cite[Def~3.5]{HuangWiseSeparability}
	A big-tree $S\subset T$ is \emph{lowest} if it is not properly contained in another big-tree. Similarly, a transection in $\Upsilon$ is \emph{lowest}, if the associated big-tree is lowest. 
\end{definition}


\begin{definition}
	\label{def:finite stature1}
	\cite[Def~3.7]{HuangWiseSeparability}
Let $G$ act on a tree $T$ without inversion. $G$ has \emph{finite stature} (with respect to the action $G\acts T$) if the action of $G$ on the collection of based big-trees has finitely many orbits.
\end{definition}

Note that $G$ satisfies Definition~\ref{def:finite stature1} if and only if the collection $\Upsilon$ is finite.

\begin{lem}
	\label{lem:finitely many conjugacy classes0}
	\cite[Lem~3.9]{HuangWiseSeparability}
The following are equivalent.
\begin{enumerate}
	\item $G$ has finite stature with respect to the action $G\acts T$.
	\item For each vertex group $V$ of $G$, there are finitely many $V$-conjugacy classes of infinite subgroups of $V$ that are of form $V\cap (\cap_{e\in E}\stab(e))$, where $E$ is a collection of edges in $T$.
	\item $(G,\mathcal{V})$ has finite stature in the sense of Definition~\ref{def:finite stature}, where $\mathcal{V}$ is the collection of vertex groups of $G$.
\end{enumerate}
\end{lem}

\begin{remark}
Suppose each vertex group of $G$ is word-hyperbolic, and each edge group is quasiconvex in its vertex groups. Then in view of Lemma~\ref{lem:finite depth} below, we can require the collection $E$ in item (2) of Lemma~\ref{lem:finitely many conjugacy classes0} to be finite.
\end{remark}

In general, if $G$ splits as a graph of groups in two different ways, then it is possible that $G$ has finite stature under one splitting, but not the other splitting, see \cite[Example~3.31]{HuangWiseSeparability}. However, we will
say that $G$ has finite stature without specifying the splitting, when it is unambiguous.

\subsection{Depth and Stature}
We recall a notion measuring the maximal length of an increasing sequence of big trees from \cite[Sec~3.2]{HuangWiseSeparability}. There are two variants according to whether the pointwise stabilizer of these big trees are commensurable.

\begin{definition}
	\label{def:depth}
	\cite[Def~3.10]{HuangWiseSeparability}
		Let $G$ be a group and let $\Lambda=\{H_1,\ldots,H_r\}$ be a finite collection of subgroups.  
	The \emph{commensurable depth} of $\Lambda$ in $G$, denoted $\delta_c(G,\Lambda)$, is the largest $d$, such that there is a strictly increasing chain $L_1 < \cdots < L_{d}$, where each $L_i$ is the intersection of a nonzero finite number of conjugates of elements of $\Lambda$. If there are arbitrarily long such sequences, then  define $\delta_c(G,\Lambda)=\infty$. We say $\Lambda$ has \emph{finite commensurable depth} in $G$ if $\delta_c(G,\Lambda)<\infty$.
\end{definition}



\begin{definition}
	\cite[Def~3.11]{HuangWiseSeparability}
	Let $G$ be a group and let $\Lambda=\{H_1,\ldots,H_r\}$ be a finite collection of subgroups.  The \emph{depth} of $\Lambda$ in $G$, denoted $\delta(G,\Lambda)$, is the largest  $d$, such that there is a strictly increasing chain $L_1 <\cdots < L_{d}$ satisfying the conditions where $|L_1|=\infty$, each $L_i$ is an intersection of a nonzero finite number of conjugates of elements of $\Lambda$, and $[L_{i+1}:L_i]=\infty$. If a largest $d$ does not exist, then  define $\delta(G,\Lambda)=\infty$.
\end{definition}

Note that $\delta_c(G,\Lambda)<\infty$ implies $\delta(G,\Lambda)<\infty$, but the converse may not be true \cite[Example~3.12]{HuangWiseSeparability}.


In the rest of this subsection, we return to our scenario that $G$ splits as a graph $\mathcal{G}$ of groups. Recall that we have identified vertex groups of $G$ with vertex stabilizers of a subtree $T_{\mathcal{G}}\subset T$. We assume in addition that each vertex group of $G$ is word-hyperbolic, and each edge group is quasiconvex in its associated vertex groups. Let $\mathcal{E}$ be the collection of edge groups of $G$. 

The following is proven in \cite{GMRS98} (see also \cite{HruskaWisePacking}):
\begin{lem}
	\label{lem:finite height}
	Let $\{H_1,\ldots,H_r\}$ be a collection of quasiconvex subgroups of a word-hyperbolic group $G$. Then $\{H_1,\ldots,H_r\}$ has finite height in $G$.
\end{lem}

Each big-tree is uniformly locally finite by Lemma \ref{lem:finite height}. Thus $G$ has finite stature if and only if   the following conditions both hold:
\begin{enumerate}
	\item There are finitely many $G$-orbits of big-trees in $T$;
	\item $\stab(S)$ acts cocompactly on each big-tree $S$.
\end{enumerate}	

\begin{lem}\label{lem:finite depth}
	\cite[Lemma~3.14 and Lem~3.16]{HuangWiseSeparability}
Suppose each vertex group of $G$ is word-hyperbolic, and each edge group is quasiconvex in its vertex groups. Choose a finite subtree $S\subset T$ and a vertex $v\in S$. Then $\pstab(S)\subset\stab(v)$ is quasiconvex. If $G$ has finite stature, then $\delta_c(G,\mathcal{E})<\infty$. Consequently, $\delta(G,\mathcal{E})<\infty$.
\end{lem}

Even if $G$ satisfies the assumption of Lemma~\ref{lem:finite depth}, the converse of Lemma~\ref{lem:finite depth} is not true. See \cite[Example~3.31]{HuangWiseSeparability}.

Since $\delta_c(G,\mathcal{E})<\infty$, the pointwise stabilizer of any subtree of $T$ can be expressed as an intersection of finitely many conjugates of edges groups. Thus the following two lemmas hold.

\begin{lem}
\label{lem:quasiconvex}
Suppose each vertex group of $G$ is word-hyperbolic, and each edge group is quasiconvex in its vertex groups. Let $v$ be a vertex in a subtree $S\subset T$. Then $\pstab(S)\subset\stab(v)$ is quasiconvex. In particular, for any vertex group $V$, each element in $\Upsilon_V$ is quasiconvex.
\end{lem}



\subsection{finite stature for the augmented space}
Let $X$ be a graph of nonpositively curved cube complex in
Definition~\ref{def:graph of cube complexes}. In this subsection, we discuss whether having finite stature is preserved if we \textquotedblleft augment\textquotedblright\ $X$ in a certain way. The reader is advised to proceed directly to Section~\ref{subsec_special cover} and come back when needed.

Let $G=\pi_1 X$. Then $G$ has a graph of groups decomposition. Let $T$ be the corresponding Bass-Serre tree. Then the universal cover $\widetilde X$ of $X$ is a tree of CAT$(0)$ cube complexes over $T$. Let $S\subset T$ be a big-tree and let $H_S$ be the image of the natural homomorphism $\stab(S)\to \Aut S$. Then there is an exact sequence $$1\to \pstab(S)\to \stab(S)\to H_S\to 1.$$

\begin{lem}
	\label{lem:proper action}
	Suppose each vertex group of $G$ is hyperbolic and each edge group is quasiconvex in the associated vertex groups. Suppose $\delta_c(G,\mathcal{E})<\infty$, where $\mathcal{E}$ is the collection of edge groups of $G$. Let $S$ be a big-tree. Then the image of $\beta:\stab(S)\to \Aut S$ acts properly on $S$.
\end{lem}

\begin{proof}
	Let $H$ be the image of $\beta$. Pick an arbitrary vertex $v\in S$. It suffices to show $H_v=\stab_H(v)$ is finite. Let $H_i=\pstab_H(B(v,i))$, where $B(v,i)$ is the $i$-ball in $S$ centered at $v$. Since $S$ is uniformly locally finite by Lemma~\ref{lem:finite height}, $H_i$ is of finite index in $H_v$. Suppose by contradiction that $H_v$ is infinite. Then there is an infinite sequence $0<n_1<n_2<\cdots$ such that $H_{n_{i+1}}\subsetneq H_{n_i}$. Thus $\beta^{-1}(H_{n_{i+1}})\subsetneq \beta^{-1}(H_{n_i})$. Note that $\beta^{-1}(H_{n_i})=\stab_G(S)\cap\pstab_G(B(v,n_i))$. We deduce from $\delta_c(G,\mathcal{E})<\infty$ that the sequence $\{\pstab_G(B(v,n_i))\}_{i=1}^{\infty}$ stabilizes for large $i$, hence the same holds for $\{\stab_G(S)\cap \pstab_G(B(v,n_i))\}_{i=1}^{\infty}$, which leads to a contradiction.
\end{proof}

\begin{lem}
	\label{lem:virtually splits}
	Suppose each vertex group of $G$ is hyperbolic and each edge group is quasiconvex in its vertex groups. Suppose $(G,\mathcal{V})$ has finite stature,
	where $\mathcal{V}$ is the collection of vertex groups of $G$. Then there is a finite index subgroup $J\le \stab(S)$ such that $J=\pstab(S)\times J_1$ for some $J_1\le J$.
\end{lem}

\begin{proof}
	Note that $\delta_c(G,\mathcal{E})<\infty$ by Lemma \ref{lem:finite depth}. By Lemma \ref{lem:proper action}, $H_S$ acts properly and cocompactly on a tree, thus $H_S$ has
	a free finite index subgroup $H'_S$. Let $\stab(S)'\le\stab(S)$ be the subgroup induced by $H'_S$. Pick an edge $e\subset S$ and let $\widetilde X_e\subset\widetilde X$ be the edge space over $e$. By Lemma \ref{lem:quasiconvex}, $\pstab(S)$ is quasiconvex in $\stab(\widetilde X_e)$. Choose a minimal $\pstab(S)$-invariant convex subcomplex of $Y\subset \widetilde X_e$ stabilized by $\pstab(S)$ (cf.\ Lemma~\ref{lem:superconvex core}).
	
	Let $g\in \stab(S)'$. We claim $gY$ and $Y$ are parallel. Note that $gY$ is invariant under $(\pstab(S))^g=\pstab(S)$. Let $Y'\subset Y$ be image of the CAT$(0)$ projection of $gY$ to $Y$. Then $Y'$ is $\pstab(S)$-invariant. By minimality, $Y=Y'$. Similarly, the image of the CAT$(0)$-projection of $Y$ to $gY$ is $gY$. Thus $gY$ and $Y$ are parallel. Note that when $ge=e$, since $H'_S$ is torsion-free,  $g\in \pstab(S)$, and hence $gY=Y$.
	
	Let $S_1\subset S$ be the convex hull of the $\stab(S)'$-image of $e$. Then $W=Y\times S_1$ embeds isometrically into $\widetilde X$ as a $\stab(S)'$-invariant convex subcomplex. Note that $W/\stab(S)'$ is a fiber-bundle over $S_1/H'_S$, and $S_1/H'_S$ is compact. The cubical structure of $W$ induces a holonomy on this fiber-bundle with finite holonomy group, and the lemma follows by passing to a finite cover with trivial holonomy group.
\end{proof}

For each vertex space $X_i$ of $X$, choose a collection of local-isometries $\{f_{ij}:X_{ij}\to X_i\}_{j=1}^{n_i}$ with compact domains. Attach the mapping cylinders of each $f_{ij}$ to $X$ to form a new graph of cube complexes $X_a$. Let $G_a=\pi_1 X_a$ with a graph of group decomposition induced from $X_a$. Note that  $X_a$ and $X$ are homotopy equivalent, and $G=G_a$. The Bass-Serre tree $T$ of $G$ naturally sits inside the Bass-Serre tree $T_a$ of $G_a$.

We need the following result which follows from standard facts about quasiconvex subgroups of hyperbolic groups, see \cite[Cor~3.22]{HuangWiseSeparability} for an explanation.
\begin{lem}
	\label{lem:finitely many conjugacy class}
	Let $\{H_1,\ldots,H_r\}$ be a collection of quasiconvex subgroups of the word-hyperbolic group $H$. Let $K$ be a quasiconvex subgroup of $H$. Then there are only finitely many $K$-conjugacy classes of subgroups of form $K\cap(\cap_{k=1}^{n}H^{g_k}_{k_i})$.
\end{lem}

\begin{lem}
	\label{lem:finite stature}
	Let $G$ be as in Lemma~\ref{lem:virtually splits}. Then $(G_a,\mathcal{V}_a)$ has finite stature when $\mathcal{V}_a$ is the collection of vertex groups of $G$.
\end{lem}

\begin{proof}
	Let $S_a\subset T_a$ be a big-tree and let $S=S_a\cap T$. Then $S$ is either one point, or $S$ is a big-tree of $T$. Note that $S_a$ is uniformly locally finite by Lemma~\ref{lem:finite height}. Moreover, $S_a$ is contained in the 1-neighborhood of $S$. We must show:
	\begin{enumerate}
		\item $\stab(S_a)\acts S_a$ is cocompact;
		\item there are finitely many $G_a$-orbits of big-trees of $T_a$.
	\end{enumerate}
	
	Let $J_1$ and $J$ be the groups in Lemma \ref{lem:virtually splits}. Since $\stab(S)\acts S$ is cocompact, the actions $J\acts S$ and $J_1\acts S$ are cocompact. Since each element in $J_1$ commutes with each element in $\pstab(S_a)\le \pstab(S)$, we have $\pstab(jS_a)=\pstab(S_a)$ for any $j\in J_1$. It follows from the definition of big-tree that $jS_a=S_a$ for any $i\in j_1$. Hence $J_1\acts S_a$ is cocompact. In particular, $\stab(S_a)\acts S_a$ is cocompact.
	
	For (2), since we already know there are finitely many $G$-orbits of big-trees of $T$, it suffices to show there are finitely many $G_a$-orbits of big-trees of $T_a$ such that their intersection with $T$ is exactly $S$. Let $K\subset S$ be a finite subtree such that $J_1K$ covers $S$ and $\pstab(K)=\pstab(S)$ (this is possible by Lemma~\ref{lem:finite depth}). Pick a big-tree $S'_a$ of $T_a$ such that $T\cap S'_a=S$. Let $K'_a$ be the union of $K$ together with all the edges of $S'_a$ that intersect $K$. We claim $\pstab(S'_a)=\pstab(K'_a)$. By a similar argument as before, $J_1$ stabilizes $S'_a$. Thus $J_1 K'_a$ covers $S'_a$. Moreover, for each $j\in J_1$, since $j$ commutes with each element in $\pstab(K'_a)\le \pstab(K)=\pstab(S)$, 
	we have $\pstab(K'_a)=\pstab(jK'_a)$. Thus the claim follows. We also observe that there are only finitely many $\pstab(K)$-conjugacy of groups of form $\pstab(K')$, where $K'$ is a finite subtree of $T_a$ such that $K'\cap T=K$ (to see this, first apply Lemma~\ref{lem:finitely many conjugacy class} with $H=\stab(v)$ for each vertex $v\in K$ to obtain a collection of quasiconvex subgroups of $\pstab(K)$ which fall into finitely many $\pstab(K)$-conjugacy classes, then apply Lemma~\ref{lem:finitely many conjugacy class} with $H=\pstab(K)$). Thus (2) follows.
\end{proof}

\section{The depth reducing quotient}
\label{subsec_quotient}
Let $G$ be the fundamental group of a finite graph of groups with the underlying graph $\G$. Let $\{V_i\}_{i=1}^{n}$ be the collection of vertex groups of $G$. Suppose there is an edge $E$ between $V_i$ and $V_j$ (we allow $i=j$). Then $E$ induces an isomorphism $\alpha_E:E_i\to E_j$ from a subgroup of $V_i$ to a subgroup of $V_j$. Note that $\alpha_E$ is a transfer isomorphism discussed after Definition~\ref{def:big-tree}.

Define the notion of \emph{quotient of graphs of groups} as follows: Let $\{q_i:V_i\to \bar{V}_i\}_{i=1}^n$ be a collection of quotient maps. They are \emph{compatible} if for any edge $E$ between $V_i$ and $V_j$, we have $\alpha_E(E_i\cap \ker q_i)=E_j\cap \ker q_j$. In this case, $\alpha_E$ descends to an isomorphism $\bar{\alpha}_E:\bar{E}_i\to\bar{E}_j$, where $\bar{E}_i=q_i(E_i)$. We define a new graph of groups with the same underlying graph $\G$, vertex groups the $\bar{V}_i$'s, and edges groups as well as boundary morphisms induced by the $\alpha_E$'s. Let $\bar{G}$ be the fundamental group of this new graph of groups. Then there is a quotient homomorphism $G\to \bar{G}$,  sending vertex groups (resp. edge groups) of $G$ to vertex groups (resp. edge groups) of $\bar{G}$. The following result  enables us to make such a quotient of $G$ with decreasing depth. See \cite[Prop~4.2]{HuangWiseSeparability}.
\begin{prop}
	\label{prop:quotients}
Let $G$ be the fundamental group of a finite graph of groups and let $\mathcal{V}$ and $\mathcal{E}$ be the collection of vertex groups and edge groups of $G$. Suppose
\begin{enumerate}
	\item each vertex group is word-hyperbolic and virtually compact special;
	\item each edge group is quasiconvex in its vertex groups;
	\item $G$ has finite stature, and $\delta(G,\mathcal{E})>0$ (note that $\delta(G,\mathcal{E})<\infty$ by Lemma~\ref{lem:finite depth}).
\end{enumerate}
Then there is a collection of quotient homomorphisms $\{\phi_V:V\to\bar{V}\}_{V\in \mathcal{V}}$ such that
\begin{enumerate}
	\item \label{conclusion10} $\bar{V}=V/\nclose{\{L^V_i\}}$, where each $L^V_i$ is a finite index subgroup of a lowest transection of $G$ in  $V$, and the collection varies over representatives of all such lowest transections; moreover, we can assume each $L^V_i$ is contained in a given finite index subgroup of its associated lowest transection;
	\item \label{conclusion20} for any edge group $E\to V$, the subgroup $\ker(E\to\bar{V})$ is generated by $V$-conjugates of $\{L^V_i\}$ that are contained in $E$;
	\item \label{conclusion2} the collection $\{\phi_V:V\to\bar{V}\}_{V\in \mathcal{V}}$ is compatible, hence there is a quotient of graphs of groups $\phi:G\to\bar{G}$ as above;
	\item \label{conclusion1} $\bar{V}$ is word-hyperbolic and virtually compact special for each $V$;
	\item \label{conclusion3} each edge group of $\bar{G}$ is quasiconvex in the corresponding vertex groups;
	\item \label{conclusion4} $(\bar{G},\bar{\mathcal{V}})$ has finite stature;
	\item \label{conclusion5} $\delta(\bar{G},\bar{\mathcal{E}})<\delta(G,\mathcal{E})$.
\end{enumerate}
\end{prop}

The following result describes how transections change under the above quotient $G\to \bar G$.
See \cite[Lem~4.9]{HuangWiseSeparability}.

\begin{lem}
	\label{lem:used later}
Let $\phi:G\to \bar{G}$ and $\phi_T:T\to\bar{T}$ be as above. Then
\begin{enumerate}
	\item for any subtree $S\subset T$ with $\pstab(S)$ infinite, $\phi(\pstab(S))$ is commensurable to $\pstab(\bar{S})$ for $\bar{S}=\phi_T(S)$.
	\item Pick finite subtree $\bar{S}\subset\bar{T}$ such that $\pstab(\bar{S})$ is infinite and pick a vertex $\bar{w}\subset\bar{S}$. Then for any $w\in T$ with $\phi_T(w)=\bar{w}$, there exists a subtree $S\subset T$ containing $w$ such that $\phi(\pstab(S))$ is commensurable to $\pstab(\bar{S})$ and $S$ is a lift of $\bar{S}$.
	\item Let $v\in T$ be a vertex and let $H_1$ and $H_2$ be two transections in $\stab(v)$. If $|\phi(H_1)\cap \phi(H_2)|=\infty$, then there exists $g\in \ker(\stab(v)\to\stab(\bar v))$ such that $\phi(H^g_1\cap H_2)=\phi(H_1)\cap \phi(H_2)$ up to finite index subgroups.
\end{enumerate}
\end{lem}

\section{Small cancellation quotients}
\label{subsec:realizing B(6)}

This section is devoted to the proof of Proposition~\ref{prop:graded}, which is a main technical ingredient towards virtual specialness.

\begin{definition}
A map between cube complexes is \emph{cubical} if it is cellular,
and its restriction to each cube is a map $\sigma\to\tau$ that factors as $\sigma\to\eta\to\tau$, where the first map $\sigma\to\eta$ is a projection onto a face of $\sigma$ and the second map $\eta\to\tau$ is an isometry.
\end{definition}

Let $Y$ be a nonpositively curved cube complex. A \emph{thickening} of $Y$ is a nonpositively curved cube complex $Y'$ such that there exists a  local-isometry $i:Y\to Y'$ which induces an isomorphism of the fundamental group. Note that the thickening map $i:Y\to Y'$ is necessarily injective. Moreover, if we pass to the universal covers, the cubical projection (\cite[Lem~13.8]{HaglundWiseSpecial}) from $\widetilde Y'$ to $\widetilde Y\subset\widetilde Y'$ is $\pi_1 Y$-equivariant, and hence induces a cubical retraction map $Y'\to Y$ that is  a homotopy equivalence.

Let $f:Y\to X$ be a local-isometry between nonpositively curved cube complexes. A \emph{thickening} of $f:Y\to X$ is a local-isometry $f':Y'\to X$ 
where $f=f'\circ i$ and $i: Y\to Y'$ is a thickening of $Y$.

\begin{prop}
	\label{prop:graded}
	Let $X$ be a compact nonpositively curved cube complex that splits as a graph of nonpositively curved cube complexes. Suppose the fundamental group of each vertex space is word-hyperbolic, and $(G,\mathcal{V})$ has finite stature, where $G=\pi_1 X$ and $\mathcal{V}$ is the collection of vertex groups of $G$.
	
	Let $\{X_V\}_{V\in \mathcal{V}}$ be the vertex spaces of $X$. Then there exists a collection of quotient maps $\{\phi_V:V\to\bar V\}_{V\in\mathcal{V}}$ such that
	\begin{enumerate}
		\item $\pi_1\bar{V}$ is virtually compact special and word-hyperbolic for each $V\in\mathcal{V}$;
		\item the collection $\{\phi_V:V\to\bar V\}_{V\in\mathcal{V}}$ is compatible, hence they induce a quotient of graphs of groups $G\to\bar{G}$;
		\item each edge group of $\bar{G}$ is finite;
		\item the cover of each $X_V$ corresponding to $\ker \phi_V$ is a special cube complex;
		\item each $\phi_V$ is induced by a $C'(\frac{1}{24})$ graded small cancellation presentation $X^{\ast}_V=\langle X_V \mid \R_V\rangle$ such that $\R_V$ is finite and $\widetilde X^{\ast}_V$ has well-embedded cones; moreover, the following hold for each edge space $X_E\to X_V$ of $X_V$:
		\begin{enumerate}
			\item there is a finite regular cover $\dot X_E\to X_E$, and a relator $Z_j\in \R_V$ such that $Z_j\to X_V$ is a thickening of $\dot X_E\to X_V$;
			\item the natural map $E/\nclose{\dot E}\to \phi_V(E)$ is an isomorphism, where $E=\pi_1 X_E$ and $\dot E=\pi_1 \dot X_E$.
		\end{enumerate}
	\end{enumerate}
\end{prop}

\subsection{Setting up and the choice of relators}
\label{subsec:setting up}
Let $\Upsilon$ and $\Upsilon_V$ be the collection of transections as in Definition~\ref{defn:phi}. We add the trivial subgroup to each of these collections. Each element of $\Upsilon_V$ is quasiconvex in $V$ by Lemma~\ref{lem:quasiconvex}.

\begin{definition}[Choosing Principal Relators]
	\label{const:choice of relators}
	For each  $L\in\Upsilon_V$,  choose a convex subcomplex $\widetilde Y_L\subset \widetilde X_V$ such that:
	\begin{enumerate}
		\item $\widetilde Y_L$ is superconvex (cf.\ Lemma~\ref{lem:superconvex core});
		\item $L$ acts cocompactly on $\widetilde Y_L$; 
		\item if $L$ is a finite index subgroup of a conjugate of $E$ in $V$ for some edge group $E\to V$, then there is a lift $\widetilde{X}_E\to \widetilde{X}_V$ with $\widetilde{X}_E\subset\widetilde{Y}_L$, where $X_E\to X_V$ is the edge space associated with $E$.
	\end{enumerate}
For each $L\subset\Upsilon_V$, we associate a locally isometric embedding $Y_L\to X_V$ where $Y_L=\widetilde Y_L/L$. Let $\R^V$ be the collection of representatives of isomorphism classes of principal components of multiple fiber-products of $\{Y_L\to X_V\}_{L\in\Upsilon_V}$. Elements in $\R^V$ are called \emph{relators}. Each element in $\R^V$ is superconvex as intersections of superconvex subspaces in $\widetilde X_V$ is superconvex.
\end{definition}

Each relator in $\R^V$ gives rise to a subgroup of $V$ by choosing a basepoint in the relator. The collection of all such subgroups is denoted by $\Upsilon^V$. 

$\R^V$ is non-empty. Actually, it follows from the definition of $\Upsilon_V$, Remark~\ref{rem:multiple fiberproduct similar}, Definition~\ref{defn:principal}, Lemma~\ref{lem:symmetric characterization} and Lemma~\ref{lem:commensurator} that there is a 1-1 correspondence between elements in $\Upsilon^V$ and commensurable classes of $\Upsilon_V$ in $V$ (each element in $\Upsilon^V$ is the minimal element in a commensurable class of $\Upsilon_V$). For the purpose of arranging Proposition~\ref{prop:graded}.(4), we assume in addition that each element in $\Upsilon^V$ is contained in a finite index normal subgroup $V'\le V$ corresponding to a special finite cover of $X_V$ (this can be arranged by first passing to appropriate finite covers of relators in $\R^V$, then symmetrize them again).

$\R^V$ has a partial order $\prec$ as follows. For  relators $Y_1,Y_2\in\R^V$, declare $Y_1\prec Y_2$ if $Y_1\to X_V$ factors through $Y_2\to X_V$. $\Upsilon^V$ also has a partial order $\prec$, where $L_1,L_2\in\Upsilon^V$ satisfies $L_1\prec L_2$ if there exists $g\in V$ such that $(L_1)^g< L_2$. These two partial orders are consistent.

By Lemma~\ref{lem:stable1}, $\R^V$ is stable. Let $Y_k,Y_i$ and $Y_j$ be as in Definition~\ref{def:stable}.\eqref{intersection}. Then there is $\kappa=\kappa(\R^V)<\infty$ independent of $Y_i$ and $Y_j$ such that $\widetilde Y_i\cap\widetilde Y_j\subset N_{\kappa}(\widetilde Y_k)$ where $\widetilde Y_i,\widetilde Y_j$ and $\widetilde Y_k$ are suitable copies of universal covers of $Y_i,Y_j$ and $Y_k$ inside $\widetilde X_V$. Later in the proof we will replace each element in $\R^V$ by a finite cover to obtain a new stable collection $\R^V_n$, however, we can take $\kappa(\R^V_n)=\kappa(\R^V)$.


\subsection{The induction hypothesis}
\label{subsubsec:induction}
Our plan is to successively replace elements in each $\R^V$ by their finite covers such that the conclusion of Proposition~\ref{prop:graded} is satisfied. This will be done inductively. Let $G=G_1$ and $\Upsilon^V=\Upsilon^V_0$. At step $k$, we will define a quotient of graphs of groups $\phi_k:G_k\to G_{k+1}$ using Proposition~\ref{prop:quotients}, and replace each group in $\{\Upsilon^V_{k-1}\}_{V\in\mathcal{V}}$ by its finite index subgroup to obtain $\{\Upsilon^V_{k}\}_{V\in\mathcal{V}}$. Suppose we have completed the first $n$ steps and obtained the following sequences.
\begin{align*}
& G=G_1\stackrel{\phi_1}{\to} G_2\stackrel{\phi_2}{\to} G_3\to\cdots\to G_n\stackrel{\phi_n}{\to} G_{n+1}\\
& \Upsilon^V=\Upsilon^V_0\to \Upsilon^V_1\to \Upsilon^V_2\to\cdots\to \Upsilon^V_{n-1}\to\Upsilon^V_n
\end{align*}

 We need several notations to describe our inductive hypothesis. For $1\le k\le n$, let $\varphi_k=\phi_{k}\circ\cdots\circ\phi_1:G_1\to G_{k+1}$. Let $\{V_k\}_{V\in\mathcal{V}}$ be the vertex groups of $G_k$. Let $\Upsilon^V_{k,<k}$ be the collection of elements in $\Upsilon^V_k$ whose $\varphi_{k-1}$-images are finite. When $k=1$,  define $\Upsilon^V_{k,<k}=\emptyset$. Let $\{L^V_{ki}\}$ be the minimal elements in $\Upsilon^V_k$ (with respect to the partial order $\prec$) whose $\varphi_{k-1}$-images are infinite (if $k=1$, then $\varphi_0$ is the identity map). It is possible that $\{L^V_{ki}\}$ is the empty collection for some $V\in\mathcal{V}$). Let $\Upsilon^V_{k,\le k}=\Upsilon^V_{k,<k}\cup \{L^V_{ki}\}$. Note that both $\Upsilon^V_{k,\le k}$ and $\Upsilon^V_{k,< k}$ are ideals of $\Upsilon^V_k$ with respect to $\prec$. Recall that an \emph{ideal} of a poset $P$ is a subset $I\subset P$ such that $y\in I$ whenever $y\prec x$ for some $x\in I$. Let $\R^V_k$ be the collection of finite covers of elements in $\R^V$ induced by $\Upsilon^V_k$. We define $\R^V_{k,<k}$ and $\R^V_{k,\le k}$ similarly.

Suppose that for each $V\in\mathcal{V}$,
\begin{enumerate}[label=(\alph*)]
	\item \label{id1} $\R^V_n$ is stable, and each element of $\R^V_n$ is a finite cover of an element of $\R^V$;
	\item \label{id2} $V_{n+1}=V_1/\nclose{\Upsilon^V_{n,\le n}}$;
	\item \label{id3} there is a grading of elements in $\R^V_{n,\le n}$ such that $\langle X_V\mid \R^V_{n,\le n}\rangle$ is a $C'(\frac{1}{24})$ graded metric small cancellation presentation;
	\item \label{id4} for $Y,Z\in \R^V_{n,\le n}$ with $Z\prec Y$, we have $\diameter(Z)+\kappa<\frac{1}{24}\systole{Y^\subpres}$ where $\kappa$ is defined at the end of Section~\ref{subsec:setting up};
	\item \label{id5} $\R^V_n$ has the relator embedding property (Definition~\ref{def:relator embedding property}) relative to $\R^V_{n,< n}$;
	\item \label{id6} $G_{n+1}$ satisfies the assumptions of Proposition~\ref{prop:quotients}, and $\delta(G_{n+1},\mathcal{E}_{n+1})\le \delta(G_1,\mathcal{E}_1)-n$, where $\mathcal{E}_i$ is the collection of edges groups of $G_i$;
	\item \label{id7} Lemma~\ref{lem:used later} holds for the quotient $\varphi_n:G_1\to G_{n+1}$.
\end{enumerate}

\begin{definition}
	\label{def:relator embedding property}
A subset $\R'$ of $\R^V_n$ has the \emph{relator embedding property} relative to a subset $\R''\subset\R^V_n$ if for any $Y_2\in\R'$ and $Y_1\in\R''$ such that $Y_1\to X_V$ factors through $Y_2\to X_V$ via $Y_1\to Y_2\to X_V$, we have that $Y_1\to Y_2$ is an embedding.
\end{definition}





We now define elements in $\Upsilon^V_{n+1}$ and $\phi_{n+1}:G_{n+1}\to G_{n+2}$. 

\subsection{Elements with finite $\varphi_n$-images}
\label{subsec:Pi}

\begin{lem}
	\label{lem:liftable shells}
Let $Y\in \R^V_n\smallsetminus \R^V_{n,\le n}$. Then $Y^{\ast}=\langle Y\mid Y\otimes_{X_V} \R^V_{n,\le n}\rangle\to \langle X_V\mid \R^V_{n,\le n}\rangle$ has liftable shells.
\end{lem}

\begin{proof}
We need to verify the assumption of Lemma~\ref{lem:liftable shell criterion}.	
Let $Y_1\in \R^V_{n,\le n}$. Let $K$ be a component of $Y_1\otimes_{X_V} Y$. If $K$ is not a copy of $Y_1$, then there exists $Y_2\prec Y_1$ such that $Y_2\to X_V$ factors through $K\to X_V$ and $\pi_1 Y_2$ is finite index in $\pi_1 K$. It follows from $Y_2\prec Y_1$ that $Y_2\in \R^V_{n,<n}$. Then $Y_2\to K$ is an embedding, since the compositions $Y_2\to K\to Y_1$ and $Y_2\to K\to Y$ are embeddings by inductive hypothesis~\ref{id5}. It follows that  $\pi_1 Y_2\to\pi_1 K$ is an isomorphism, hence $\pi_1 K^*=1$ where $K^*$ is defined in Lemma~\ref{lem:liftable shell criterion}. Since $K\subset N_\kappa(Y_2)$, $\diameter(K)\le \frac{1}{2}\systole{Y^\subpres_1}$ by inductive hypothesis~\ref{id4}. Now the lemma follows from Lemma~\ref{lem:liftable shell criterion}.
\end{proof}

\begin{lem}
	\label{lem:embedding}
Let $Y\in \R^V_n\smallsetminus \R^V_{n,\le n}$ and $Y^*$ be as in Lemma~\ref{lem:liftable shells}. Then $Y^*$ has a finite cover $\widehat Y^*$ whose cubical part $\widehat Y$ has the relator embedding property relative to $\R^V_{n,\le n}$. 
\end{lem}

\begin{proof}
By inductive hypotheses~\ref{id3}~and~\ref{id4} and Lemma~\ref{lem:well-embedded cones}, cones of $\widetilde X_V^*$ are embedded where $\widetilde X^*_V$ is the universal cover of $X^*_V=\langle X_V\mid \R^V_{n,\le n}\rangle$. It follows that for any $Z\prec Y$ with $Z\in \R^V_{n,\le n}$, the lift $Z\to \widetilde Y^*$ is an embedding (since $Z\to \widetilde Y^*\to \widetilde X_V^*$). By Lemmas~\ref{lem:quasi-isom embedding}~and~\ref{lem:liftable shells}, $\pi_1 Y^*$ is a subgroup of $V_{n+1}$. By inductive hypothesis~\ref{id6}, $V_{n+1}$ is virtually compact special. Thus $\pi_1 Y^*$ is residually finite. Hence there is a finite cover $\widehat Y^*\to Y^*$ such that 
 $Z\to \widehat Y$ is an embedding
whenever $Z\to X_V$ factors through $\widehat Y\to X_V$.
\end{proof}

\begin{lem}
	\label{lem:sym0}
We replace each $Y\in \R^V_n\smallsetminus \R^V_{n,\le n}$ by $\widehat Y$ as in Lemma~\ref{lem:embedding} (elements in $\R^V_{n,\le n}$ remain unchanged) to obtain a new collection $\widehat \R^V_n$. Let $\mathring \R^V_n$ be the principal components of $\widehat \R^V_n$. Then 
\begin{enumerate}
	\item $\R^V_{n,\le n}\subset \mathring \R^V_n$;
	\item $\mathring \R^V_n\smallsetminus \R^V_{n,\le n}$ has the relator embedding property relative to $\R^V_{n,\le n}$.\\
	Thus by inductive hypothesis~\ref{id5}, $\mathring \R^V_n$ has the relator embedding property relative to $\R^V_{n,\le n}$. 
\end{enumerate} 
\end{lem}

\begin{proof}
Recall that each element in $\mathring \R^V_n$ is a component in the multiple fiber-product of elements in $\widehat \R^V_n$. Now (1) follows from the stability of $\R^V_n$  and the fact that for any $Z\in \R^V_{n,\le n}$ and $Y\in \R^V_n\smallsetminus \R^V_{n,\le n}$, whenever $Z\to X_V$ factors through $Y\to X_V$, then $Z\to Y$ factors through $\widehat Y\to Y$. (2) follows from Lemma~\ref{lem:embedding} and the definition of fiber-products. 
\end{proof}

The argument proving Lemma~\ref{lem:liftable shells}
shows that $\mathring Y^{\ast}=\langle \mathring Y\mid \mathring Y\otimes_{X_V} \R^V_{n,\le n}\rangle\to \langle X_V\mid \R^V_{n,\le n}\rangle$ has liftable shells for each $\mathring Y\in \mathring\R^V_n\smallsetminus \R^V_{n,\le n}$.

Let $\mathring \R^V_{n,<n+1}$ be the collection of elements in $\mathring\R^V_n$ whose fundamental groups have finite $\varphi_n$-images. By inductive hypothesis~\ref{id2}, $\R^V_{n,\le n}\subset \mathring\R^V_{n,<n+1}$. Let $\mathring\R^V_{n,<n+1}\smallsetminus \R^V_{n,\le n}=\{\mathring Y_1,\mathring Y_2,\ldots,\mathring Y_m\}$ such that for each $1\le j\le m$, $\R^V_{n,\le n}\cup \{\mathring Y_1,\mathring Y_2,\ldots,\mathring Y_j\}$ is an ideal in $\mathring\R^V_n$. For each $\mathring Y_i$, it follows from Lemma~\ref{lem:liftable shells} that $\pi_1 \mathring Y^*_i=\varphi_n(\pi_1 \mathring Y_i)$ is finite, where $\mathring Y^*_i=\langle \mathring Y_i\mid \mathring Y_i\otimes_{X_V} \R^V_{n,\le n}\rangle$. Let $\dot Y_i$ be the cubical part of the universal cover of $\mathring Y^*_i$. Assign grades to elements in $\{\dot Y_1,\dot Y_2,\ldots,\dot Y_m\}$ so that $\grade(\dot Y_1)<\cdots<\grade(\dot Y_m)$ and $\grade(\dot Y_1)$ is bigger than the grade of any element in $\R^V_{n,\le n}$.

\begin{lem}\
	\label{lem:replace1}
\begin{enumerate}
	\item \label{replace11} $V_{n+1}=V_1/\nclose{\Upsilon^V_{n,\le n}}=V_1/\nclose{\Upsilon^V_{n,\le n},\dot Y_1,\ldots,\dot Y_m}$.
	\item \label{replace12} The statements of inductive hypotheses \ref{id3}~and~\ref{id4} continue to hold with $\R^V_{n,\le n}$ replaced by $\R^V_{n,\le n}\cup\{\dot Y_1,\ldots,\dot Y_m\}$.
	\item \label{replace13} $\dot \R^V_n$ is stable where $\dot \R^V_n$ is obtained by replacing each $Y_i$ by $\dot Y_i$ in $\mathring\R^V_n$.
\end{enumerate}
\end{lem}

\begin{proof}
For simplicity of notation, we only prove the case $m=1$. The general case follows from the same argument and induction on $m$. (1) is clear. To see (2), note that adding the relator $Y_1$ may bring more pieces to $Y\in \R^V_{n,\le n}$. However, by the same argument as in Lemma~\ref{lem:liftable shells}, if a component $K$ of $Y_1\otimes_{X_V} Y$ is not a copy of $Y$, then there exists $Z\prec Y$ such that $Z\to K$ is an embedding and $\pi_1 Z\to\pi_1 K$ is an isomorphism. Hence $K\subset N_\kappa(Y_2)$, and $K$ lifts to $\widetilde Y^{\subpres}$. Thus $|K|_{Y^\subpres}\le \frac{1}{24}\systole{Y^{\subpres}}$ by inductive hypothesis~\ref{id4}. Note that $\systole{\dot Y^{\subpres}_1}=\infty$ (since $\pi_1\dot Y^{\subpres}_1=\pi_1 \langle \dot Y_1\mid \dot Y_1\otimes_{X_V} \R^V_{n,\le n}\rangle=1$). Now (2) follows.

To prove (3), we use Lemma~\ref{lem:stable}. Let $H=\C_V(\pi_1 \dot Y_1)=\C_V(\pi_1 Y_1)$. Note that $\pi_1\dot Y_1=(\varphi_n|_{\pi_1 Y_1})^{-1}(\{1\})=\pi_1 Y_1\cap (\varphi_n|_H)^{-1}(\{1\})$, where $\{1\}$ denotes the trivial subgroup of $V_{n+1}$. As $\{1\}\trianglelefteq \varphi_n(H)$
we have $(\varphi_n|_H)^{-1}(\{1\})\trianglelefteq H$. It follows that $\pi_1\dot Y_1\trianglelefteq H$. Thus Lemma~\ref{lem:stable}.(1) is satisfied. Lemma~\ref{lem:stable}.(2) follows from the definition of $\dot Y_1$. Hence (3) follows. 
\end{proof}

Since $\dot Y^{\subpres}_i$ has trivial fundamental group, by Lemma~\ref{lem:replace1} and the argument in Lemma~\ref{lem:embedding}, $Z\to \dot Y_i$ is an embedding whenever $Z \to X_V$ factor through $Y_i\to X_V$. Now we can repeat the process in Lemma~\ref{lem:embedding} and Lemma~\ref{lem:sym0} to ensure that $\dot R^V_n$ has the relator embedding property relative to $\R^V_{n,\le n}\cup\{\dot Y_1,\ldots,\dot Y_m\}$.
\subsection{Minimal elements with infinite $\varphi_n$-images}
\label{subsec:Lambda}

Let $\{\dot L^V_{n+1,i}\}$ be the minimal elements in $\dot \Upsilon^V_n$ whose $\varphi_n$-images are infinite. Let $\dot Y^V_{n+1,i}$ be the associated relators.

\begin{lem}
	\label{lem:commensurable}
	For each lowest transection $T\in V_{n+1}$, there exists at least one element of $\{\dot L^V_{n+1,i}\}$ such that its $\varphi_n$-image is commensurable to $T$ in $V_{n+1}$. For each $\dot L^V_{n+1,i}$ the image $\varphi_n(\dot L^V_{n+1,i})$ is commensurable with a lowest transection of $V_{n+1}$.
\end{lem}

\begin{proof}
	By inductive hypothesis~\ref{id6}, $G_{n+1}$ has finite stature. Hence any transection of $G_{n+1}$ can be expressed as the pointwise stabilizer of a finite big-tree of $G_{n+1}$ by Lemma~\ref{lem:finite depth}. Suppose $L\in \dot\Upsilon^V_{n+1}$ is minimal among all elements such that $\varphi_n(L)$ is commensurable to $T$ in $V_{n+1}$ (by inductive hypothesis~\ref{id7}, there exists at least one such element in $\dot\Upsilon^V_{n+1}$). We claim that $\varphi_n(L')$ is finite for any $L'\in\dot\Upsilon^V_{n+1}$ and $L'\prec L$. Otherwise, by inductive hypothesis~\ref{id7}, $\varphi_n(L')$ is commensurable in $V_{n+1}$ to a transection of $G_{n+1}$. Since $\varphi_n(L)$ is commensurable to a lowest transection of $G_{n+1}$ and $\varphi_n(L')\prec \varphi_n(L)$,  we have $\varphi_n(L')$ and $\varphi_n(L)$ are commensurable in $V_{n+1}$, which contradicts the minimality of $L$.
	
	We prove the second statement. By inductive hypothesis~\ref{id7}, the image $\varphi_n(\dot L^V_{n+1,i})$ is commensurable in $V_{n+1}$ to a transection in $V_{n+1}$. Then there is a lowest transection $T\le V_{n+1}$ such that it has a finite index subgroup contained in $\varphi_n(\dot L^V_{n+1,i})$. We assume without loss of generality that $T\le \varphi_n(\dot L^V_{n+1,i})$. 
	
	By inductive hypothesis \ref{id7}, there exists $L\in \dot\Upsilon^V_n$ such that $\varphi_n(L)$ is commensurable to $T$. Since $\dot L^V_{n+1,i}$ and $L$ are finite index in some transections of $V$, by inductive hypothesis~\ref{id7} and Lemma~\ref{lem:used later}.(3), there exists $h\in V$ such that $\varphi_n((\dot L^V_{n+1,i})^h\cap L)$ is commensurable to $\varphi_n(\dot L^V_{n+1,i})\cap \varphi_n(L)$, which is commensurable to $T$. By minimality of $\dot L^V_{n+1,i}$, we have $(\dot L^V_{n+1,i})^h\cap L=(\dot L^V_{n+1,i})^h$.
	\end{proof}

Let $\dot \R^V_{n,< n+1}=\R^V_{n,\le n}\cup\{\dot Y_1,\ldots,\dot Y_m\}$. Let $\dot \R^V_{n,\le n+1}=\dot\R^V_{n,<n+1}\cup\{\dot \R^V_{n+1,i}\}$. 

By minimality of each $\dot Y^V_{n+1,i}$, it is impossible that $\dot Y^V_{n+1,i}\prec \dot Y^V_{n+1,i'}$, so we assign the same grade to each element in $\{\dot Y^V_{n+1,i}\}$ so that 
their grades are bigger than the grade of any element in $\dot \R^V_{n,< n+1}$. Then $(\dot Y^V_{n+1,i})^{\subpres}=(\dot Y^V_{n+1,i})^*=\langle \dot Y^V_{n+1,i}\mid \dot Y^V_{n+1,i}\otimes_{X_V} \dot \R^V_{n,< n+1}\rangle$. Note that $Y\in \dot \R^V_{n,< n+1}$ whenever $Y\in\dot R^V_n$ satisfies $Y\prec \dot Y^V_{n+1,i}$. Then we can argue as Lemma~\ref{lem:liftable shells} to see $(\dot Y^V_{n+1,i})^{\ast}\to \langle X_V\mid \dot \R^V_{n,< n+1}\rangle$ has liftable shells. Thus $\varphi_n(\pi_1 \dot Y^V_{n+1,i})=\pi_1 (\dot Y^V_{n+1,i})^*=\pi_1 (\dot Y^V_{n+1,i})^{\subpres}$ by Lemma~\ref{lem:quasi-isom embedding}. In particular, $\pi_1 (\dot Y^V_{n+1,i})^{\subpres}$ is residually finite. 

Pick $Z\in \dot \R^V_{n,\le n+1}$ and let $K$ be a component of $Z\otimes_{X_V} \dot Y^V_{n+1,i}$. We argue as in Lemma~\ref{lem:liftable shells} to deduce that either $K$ is isomorphic to $Z$, or there is $W\in \dot \R^V_{n,< n+1}$ such that $W\prec Z$ and $W\prec \dot Y^V_{n+1,i}$ and $W\to K$ is an embedding and $\pi_1 W\to \pi_1 K$ is an isomorphism. Thus $|K|_{(\dot Y^V_{n+1,i})^{\subpres}}\le \diameter(W)+\kappa$. Let
\begin{center}
$M=\max\{\diameter(\mathrm{wall\ pieces\ of\ }\dot \R^V_{n,\le n+1}),\ \diameter(W)  : W\in \dot \R^V_{n,< n+1}\}$.
\end{center}
By Lemma~\ref{lem:bound on wall pieces}, $M<\infty$ since the lift of each relator is superconvex in $\widetilde X_V$. 
 
\begin{definition}[Choice of $\bar Y^V_{n+1,i}$ and $\bar L^{V_{n+1}}_i$]
	\label{def:choice}
For each $\dot Y^V_{n+1,i}$, we choose a finite cover $(\check Y^V_{n+1,i})^{\subpres}\to (\dot Y^V_{n+1,i})^{\subpres}$ such that $\systole{(\check Y^V_{n+1,i})^{\subpres}}> 24 (M+\kappa)$. For each $V\in\mathcal{V}$, we choose a finite index subgroup of each lowest transection in $V_{n+1}$ to form the collection $\{\bar L^{V_{n+1}}_i\}$ such that
\begin{enumerate}
	\item each $\pi_1(\check Y^V_{n+1,i})^{\subpres}=\varphi_n(\pi_1 \check Y^V_{n+1,i})$ contains some $\bar L^{V_{n+1}}_i$ as its finite index subgroup (up to conjugacy in $V_{n+1}$);
	\item each $\bar L^{V_{n+1}}_i$ is normal in $\C_{V_{n+1}}(\bar L^{V_{n+1}}_i)$;
	\item quotienting by these $\{\bar L^{V_{n+1}}_i\}$ gives rise to a quotient of graphs of groups $\phi_{n+1}:G_{n+1}\to G_{n+2}$ satisfying Proposition~\ref{prop:quotients}.
\end{enumerate}
(1) is possible by Lemma~\ref{lem:commensurable}. (2) and (3) can be arranged because of Proposition~\ref{prop:quotients}.\eqref{conclusion10} (if (2) is not satisfied, then we can  replace $\bar L^{V_{n+1}}_i$ by its normal closure in $\C_{V_{n+1}}(\bar L^{V_{n+1}}_i)$, this does not change the quotient map $\phi_{n+1}$).

Let $\{(\bar Y^V_{n+1,i})^{\subpres}\to (\check Y^V_{n+1,i})^{\subpres}\}$ be the finite covers induced by $\{\bar L^{V_{n+1}}_i\}$. We define $\R^V_{n+1}$ by replacing each $\dot Y^V_{n+1,i}$ by $\bar Y^V_{n+1,i}$ in $\dot \R^V_n$.
\end{definition}

Now we show the inductive hypothesis in Section~\ref{subsubsec:induction} holds with the above choice of $\R^V_{n+1}$. Recall that $\dot \R^V_n$ is stable, 
so the verification of \ref{id1} is similar to the proof of Lemma~\ref{lem:replace1}.\eqref{replace13}. Here we only verify that $\pi_1 \bar Y^V_{n+1,i}\trianglelefteq\C_{V}(\pi_1 \bar Y^V_{n+1,i})=\C_{V}(\dot Y^V_{n+1,i})$. Let $H=\pi_1 \dot Y^V_{n+1,i}$ and $H_1=\C_{V}(H)$. Since $\dot \R^V_n$ is stable, $H\trianglelefteq H_1$. Suppose $\varphi_n(H)$ contains $\bar L^{V_{n+1}}_j$ as a finite index subgroup (Definition~\ref{def:choice}.(1)). Then $\pi_1 \bar Y^V_{n+1,i}=(\varphi_n|_H)^{-1}(\bar L^{V_{n+1}}_j)$. Since $\bar L^{V_{n+1}}_j\trianglelefteq\C_{V_{n+1}}(\bar L^{V_{n+1}}_j)$ by Definition~\ref{def:choice}.(2), $\bar L^{V_{n+1}}_j\trianglelefteq\varphi_n(H_1)$. Then $(\varphi_n|_{H_1})^{-1}(\bar L^{V_{n+1}}_j)\trianglelefteq H_1$. It follows that $\pi_1 \bar Y^V_{n+1,i}=(\varphi_n|_H)^{-1}(\bar L^{V_{n+1}}_j)=H\cap (\varphi_n|_{H_1})^{-1}(\bar L^{V_{n+1}}_j)$ is normal in $H_1$.

\ref{id2} is clear. \ref{id4} follows from that $\systole{\dot Y_i}=\infty$ (see Section~\ref{subsec:Pi}) and our choice of $\bar Y^V_{n+1,i}$. \ref{id5} follows from the discussion at the end of Section~\ref{subsec:Pi} (note that $\R^V_{n+1,<n+1}=\R^V_{n,\le n}\cup\{\dot Y_1,\ldots,\dot Y_m\}$). For \ref{id3}, pick $Z\in \R^V_{n+1,n+1}$, let $P$ be a piece of $Z$. We need to show $|P|_{Z^{\subpres}}<\frac{1}{24}\systole{Z}$. This is  clear if $P$ is a wall-piece. 
Suppose $P$ is a cone-piece. Then $P$ is the universal cover of a component $K$ of $Z\otimes_{X_v} Y$ for some $Y\in\R^V_{n+1,n+1}$, where $K$ is not a copy of $Z$ or $Y$. By the same argument as in Lemma~\ref{lem:replace1}.\eqref{replace12}, we deduce that $|K|_{Y^{\subpres}}< \frac{1}{24}\systole{Z}$ from \ref{id5} (when $Z=\bar Y^V_{n+1,i}$, we also use $\systole{(\bar Y^V_{n+1,i})^{\subpres}}\ge\systole{(\check Y^V_{n+1,i})^{\subpres}} > 24 (M+\kappa)$). \ref{id6} follows from Proposition~\ref{prop:quotients}, and \ref{id7} follows from Lemma~\ref{lem:used later}.

\subsection{Conclusion}
By Section~\ref{subsubsec:induction}.\ref{id6}, the induction will stop after finitely many steps and we obtain $G_{r+1}$ with finite edge groups. Let $\{\phi_k:G_k\to G_{k+1}\}_{k=1}^{r}$ and $\{\Upsilon^V_k\}_{k=1}^r$ be the sequences of quotient maps and relators produced by this process. If $\Upsilon^V_{r+1}\smallsetminus\Upsilon^V_{r,\le r}=\emptyset$, then we define $\Upsilon^V_{r+1}=\Upsilon^V_r$. If $\Upsilon^V_{r+1}\smallsetminus\Upsilon^V_{r,\le r}\neq\emptyset$, then we repeat the process in Section~\ref{subsec:Pi} to replace each element in $\Upsilon^V_{r+1}\smallsetminus\Upsilon^V_{r,\le r}$ by its appropriate finite index subgroup to define $\Upsilon^V_{r+1}$. Then there is a grading of elements in $\R^V_{r+1}$ such that $\langle X_V\mid \R^V_{r+1}\rangle$ is a $C'(\frac{1}{24})$ graded metric small cancellation presentation.

We define the map $\phi_V$ in Proposition~\ref{prop:graded} to be $\varphi_r$. Note that Proposition~\ref{prop:graded}.(4) is satisfied by our choice of $\Upsilon^V$ in Section~\ref{subsec:setting up}. By Lemma~\ref{lem:well-embedded cones}, $\widetilde X^*_V$ has well-embedded cones where $X^*_V=\langle X_V\mid \R^V_{r+1}\rangle$.

Let $X_E\to X_V$ be an edge space. Then there is a unique element $\dot X_E\in \R^V_{r+1}$ such that $\pi_1 \dot X_E$ is a finite index subgroup of $\pi_1 X_E$ (up to conjugacy in $\pi_1 X_V$). By Definition~\ref{const:choice of relators}.(3), we can replace $X_E$ by a thickening of $X_E$ such that $\dot X_E\to X_E$ is a finite cover. Let $E=\pi_1 X_E$ and $\dot E=\pi_1 \dot X_E$. Since $\dot E\trianglelefteq \C_{V}(\dot E)$, we have $\dot E\trianglelefteq E$. Conditions (5a) and (5b) of Proposition~\ref{prop:graded} follows from Lemma~\ref{lem:iso}.

\begin{lem}
	\label{lem:iso}
The map $E/\dot E\to \phi_V(E)$ is an isomorphism.
\end{lem}

\begin{proof}
Let $E_1=E$ and $E_{k+1}=\varphi_k(E_1)$. Let $\{L^V_{ki}\}$ be the collection of minimal elements in $\Upsilon^V_{r+1}$ whose $\varphi_{k-1}$-images are infinite. Then $\ker(V_k\to V_{k+1})=\nclose{\varphi_{k-1}(L^V_{ki})}_{V_k}$ by our construction. By Proposition~\ref{prop:quotients}.\eqref{conclusion20}, $\ker(E_k\to E_{k+1})$ is generated by $V_k$-conjugates of elements in $\{\varphi_{k-1}(L^V_{ki})\}$ that are contained in $E_k$. We claim that if $(\varphi_{k-1}(L^V_{ki}))^g\le E_k$ for some $g\in V_k$, then there exists $h\in V$ such that $\varphi_{k-1}((L^V_{ki})^h)=(\varphi_{k-1}(L^V_{ki}))^g$ and $(L^V_{ki})^h\le \dot E$. Indeed, as $(\varphi_{k-1}(L^V_{ki}))^g=\varphi_{k-1}((L^V_{ki})^{l})$ for some $l\in V$, we know $\varphi_{k-1}((L^V_{ki})^{l})\cap \varphi_{k-1}(E)$ is infinite. By inductive hypothesis~\ref{id6} and Lemma~\ref{lem:used later}.(3) there exists $h\in V$ such that $\phi_{k-1}(h)=\phi_{k-1}(l)$ and $\varphi_{k-1}((L^V_{ki})^h\cap E)=(\varphi_{k-1}(L^V_{ki}))^g\cap \varphi_{k-1}(E)=(\varphi_{k-1}(L^V_{ki}))^g\cap E_k=(\varphi_{k-1}(L^V_{ki}))^g$ up to finite index subgroups. In particular, $\phi_{k-1}((L^V_{ki})^h\cap E)$ is infinite, thus $(L^V_{ki})^h\le \dot E$ by minimality of $L^V_{ki}$. Moreover, $\varphi_{k-1}((L^V_{ki})^{h})=\varphi_{k-1}((L^V_{ki})^{l})=(\varphi_{k-1}(L^V_{ki}))^g$. Thus the claim holds. By this claim, $\ker(E_k\to E_{k+1})$ is generated by $\varphi_{k-1}$-images of $V$-conjugates of elements in $\{L^V_{ki}\}$ that are contained in $\dot E$. Hence $\ker(E \to E_{r+1})$ is generated by $V$-conjugates of elements in $\Upsilon^V_{r+1}$ that are contained in $\dot E$, which justifies the lemma.
\end{proof}

\section{Specialness and Stature}
\label{subsec_special cover}
In this section we prove Theorem~\ref{thm:main}.

\subsection{Criterion for specialness}
We collect several tools from \cite{HaglundWiseCoxeter} for verifying specialness.
The following is a consequence of \cite[Thm~3.5 and Cor~4.3]{HaglundWiseCoxeter}:
\begin{thm}\label{thm:shorter profinite criterion}
Let $G$ act on the special cube complex $C$ by cubical automorphisms such that:
	\begin{enumerate}
		\item $G$ acts properly on $C$.
		\item $G$ acts cocompactly on $C$.
		\item $\intersector_G(A,B)$ is separable for each pair of
		intersecting hyperplanes $A,B$, where $\intersector_G(A,B)$ is the collection of elements $g\in G$ with $A\cap gB\neq\emptyset$.
		\item $\stab_G(A)$ is separable for each hyperplane.
	\end{enumerate}
Then $G$ has a finite index subgroup $F$ such that 
\begin{enumerate}
	\item for each $g\in F$, if $g$ stabilizes a cube then $g$ pointwise stabilizes that cube, hence $C/F$ is a cube complex;
	\item $C/F$ is a special cube complex.
\end{enumerate}
\end{thm}

The following appears as \cite[Thm~5.2]{HaglundWiseCoxeter}:
\begin{thm}
	\label{thm:graph of cube complexes}
	Let $X$ split as a graph of nonpositively curved cube complexes (cf.\ Definition~\ref{def:graph of cube complexes}),
	where  each vertex
	space $X_v$ and edge space $X_e$ is special
	with finitely many hyperplanes.
	Then $X$ has a finite special cover
	provided the attaching maps of edge spaces satisfy
	the following:
	\begin{enumerate}
		\item the attaching maps $X_e\rightarrow X_{\iota(e)}$ and
		$X_e \rightarrow X_{\tau(e)}$ are injective local-isometries;
		\item distinct hyperplanes of
		$X_e$ map to distinct hyperplanes of $X_{\iota(e)}$ and
		$X_{\tau(e)}$;
		\item noncrossing hyperplanes map to noncrossing hyperplanes;
		\item no hyperplane of $X_e$ extends in $X_{\iota(e)}$ to a hyperplane dual to an edge that intersects $X_e$ in a single vertex (such a hyperplane of $X_{\iota(e)}$ is said to {\em inter-osculate $X_e$}); similarly  no hyperplane of $X_{\tau(e)}$ inter-osculates $X_e$.
	\end{enumerate}
\end{thm}
Note that condition (4) can be replaced by that the mapping cylinder of each attaching map from an edge space to a vertex space is special.

\subsection{finite stature implies specialness}
In this subsection we prove the following theorem.
\begin{thm}
	\label{thm:main2}
	Let $X$ be a graph of nonpositively curved cube complexes such that $X$ is compact. 
	Suppose the fundamental group of each vertex space is word-hyperbolic (hence each vertex space is virtually special), and $\pi_1 X$ has finite stature with respect its vertex groups. Then $X$ is virtually special.
\end{thm}

Let $X$ be as in Theorem~\ref{thm:main2} with $\pi_1X$ denoted by $G$, and the Bass-Serre tree denoted by $T$. Suppose $\{X_i\}_{i=1}^n$ are vertex spaces of $X$. Let $h_{ij}$ be the collection of hyperplanes in the vertex space $X_i$ of $X$. This gives rise to a family of local-isometries $\{f_{ij}:N(h_{ij})\to X_i\}$ from carriers of $h_{ij}$ to $X_i$. For each vertex space $X_i$, let $g_i:X_i\to X_i$ be the identity map. Attach the mapping cylinders of $f_{ij}$'s and $g_i$'s to $X$ to form a new graph of cube complexes $X_a$. There is an embedding $X\to X_a$ and we still denote the image of $\{X_i\}_{i=1}^n$ under this embedding by $\{X_i\}_{i=1}^n$. Denote the extra vertex spaces of $X_a$ by $\{X_i\}_{i=n+1}^m$. Let $G_a$ and $T_a$ be the fundamental group and Bass-Serre tree of $X_a$. Note that $G_a=G$ and $X_a$ deformation retracts to $X$. Let $\{V_i\}_{i=1}^n$ be vertex groups of $G$, and $\{V_i\}_{i=1}^m$ be vertex groups of $G_a$.

By Lemma~\ref{lem:finite stature}, $G_a$ has finite stature with respect to its vertex groups. We apply Proposition~\ref{prop:graded} to $X_a$ and $G_a$. Let $\{\phi_i:V_i\to\bar V_i\}$ be the quotient maps such that they induce $\phi:G_a\to \bar G_a$. We also use $\phi$ to denote the associated map $G\to \bar G$. Let $X^*_i=\langle X_i\mid \Upsilon_i\rangle$ be the cubical presentation that induces $\phi_i:V_i\to\bar V_i$ and satisfies the conditions of Proposition~\ref{prop:graded}. Denote the Bass-Serre tree of $\bar G_a$ and $\bar G$ by $\bar T_a$ and $\bar T$. Recall that each edge group of $\bar G_a$ is finite. Thus each vertex group of $\bar G_a$ is also finite by the construction of $G_a$, hence $\bar G_a$ is virtually free. Let $\widehat X\to X$ and $\widehat X_a\to X_a$ be the covering map induced by $\ker \phi\le G=G_a$. Then $\widehat X$ (resp. $\widehat X_a$) is a tree of cube complexes whose underlying tree is $\bar T$ (resp. $\bar T_a$).

\begin{lem}
	\label{lem:intersection of cones}
Let $\widehat X_v$ be a vertex space of $\widehat X$. Let $\{\widehat Y_j\to \widehat X_v\}_{j=1}^3$ be local-isometries such that each $\widehat Y_j$ is either the carrier of a hyperplane in $\widehat X_v$, or an edge space of $\widehat X_v$. Then the following holds.
\begin{enumerate}
	\item Each $\widehat Y_j$ embeds.
	\item If $\widehat Y_1\cap \widehat Y_2\neq\emptyset$, then $\widehat Y_1\cap \widehat Y_2$ is connected.
	\item If $\{\widehat Y_1,\widehat Y_2,\widehat Y_3\}$ pairwise intersect, then $\widehat Y_1\cap\widehat Y_2\cap\widehat Y_3\neq\emptyset$.
\end{enumerate}
\end{lem}

\begin{proof}
Let $X_i$ be the vertex space in $X$ covered by $\widehat X_v$. We also view $\widehat X_v$ (resp. $X_i$) as a vertex space of $\widehat X_a$ (resp. $X_a$). Since the covering map $\widehat X_v\to X_i$ corresponds to $\ker \phi_i$, we have $\widehat X_v$ is the cubical part of the universal cover of $X^*_i$.

Note that $\{\widehat Y_j\to \widehat X_v\}_{j=1}^3$ are edge spaces of $\widehat X_v$ in $\widehat X_a$. Let $\{Y_j\to X_i\}_{j=1}^3$ be edge spaces of $X_i$ in $X_a$ covered by $\{\widehat Y_j\to \widehat X_v\}_{j=1}^3$. By Proposition~\ref{prop:graded}.(5a), there are finite regular covers $\{\dot Y_j\to Y_j\}_{i=1}^3$ and relators $\{Z_j\}_{j=1}^3\in\Upsilon_i$ such that $Z_j$ is a thickening of $\dot Y_j$. However, Proposition~\ref{prop:graded}.(5b) implies that $\pi_1 \dot Y_j=\ker (\pi_1 Y_j\to \phi(\pi_1 Y_j))$, thus $\dot Y_j=\widehat Y_j$. Thus $Z_j\to X_i$ lifts to $Z_j\to \widehat X_v$, which is a thickening of $\widehat Y_j$. Since the universal cover $\widetilde X^*$ of $X^*$ has well-embedded cones, and each $Z_j$ is a cone of $\widetilde X^*$, Conclusions~(1), (2)~and~(3) hold for $\{Z_j\}_{j=1}^3$. By Lemma~\ref{lem:deformation retract} below, these conclusions also hold for $\{\widehat Y_j\}_{j=1}^3$.
\end{proof}

\begin{remark}
	\label{lem:intersection of cones1}
Lemma~\ref{lem:intersection of cones} also holds if each $\widehat Y_j$ is either a hyperplane of $\widehat X_v$ or an edge space of $\widehat X_v$. Indeed, since the carrier of a hyperplane is a thickening of a hyperplane, this statement can be deduced from Lemma~\ref{lem:intersection of cones} and Lemma~\ref{lem:deformation retract}.
\end{remark}

\begin{lem}
	\label{lem:deformation retract}
Let $Z$ be a nonpositively curved cube complex. Pick connected locally-convex subcomplexes $\{C_i\}_{i=1}^{3}$ and $\{C'_i\}_{i=1}^{3}$ such that each $C'_i\to Z$ is a thickening of $C_i\to Z$. Then
\begin{enumerate}
	\item if $C_1\cap C_2\neq\emptyset$ and $C'_1\cap C'_2$ is connected, then $C_1\cap C_2$ is connected;
	\item if $\{C_i\}_{i=1}^{3}$ pairwise intersect, $\cap_{i=1}^3C'_i\neq\emptyset$, and each $C'_i\cap C'_j$ is connected, then $\cap_{i=1}^3C_i\neq\emptyset$.
\end{enumerate}
\end{lem}

\begin{proof}
Our assumption implies there is a deformation retraction $r_i:C'_i\to C_i$ for each $i$. For (1), we assume by contradiction that $C_1\cap C_2$ has at least two connected components $E_1$ and $E_2$. Let $d$ be the minimal length of null-homotopic loop in the 1-skeleton that travels from a vertex $v_1\in E_1$ to a vertex $v_2\in E_2$ in $C_1$ and travel back from $v_2$ to $v_1$ in $C_2$. It is clear that $d>0$. We show there exists at least one such loop, hence $d<\infty$. Pick vertices $w_i\in E_i$ for $i=1,2$ and let $\omega\subset C'_1\cap C'_2$ be a path from $w_1$ to $w_2$. Then the concatenation of $r_1(w)$ and $r_2(w)$ gives a null-homotopic loop as required (by approximation, we can assume this loop is in the 1-skeleton). Choose a disk diagram $D$ with minimal number of squares among disk diagrams bounded by all such loops with length $=d$. Then there cannot be a spur in $\partial D$, otherwise we can find a loop of length $<d$ satisfying our conditions; and there cannot be a length 2 subpath in $\partial D$ forming the corner of a 2-cube in $D$, otherwise by local-convexity of $C_1,C_2,E_1$ and $E_2$, we can pass to a disk diagrams with fewer squares bounded by a loop of length $d$ satisfying our conditions. Thus it follows from \cite[Lem~2.8]{WiseIsraelHierarchy} that $D$ is a single point,  contradicting $d>0$.

Now we prove (2). Let $d$ be the infimal length of null-homotopic loop in the 1-skeleton that travels inside $C_2$ from a vertex $v_{12}\in C_1\cap C_2$ to a vertex $v_{23}\in C_2\cap C_3$, then travel inside $C_3$ from $v_{23}$ to a vertex $v_{31}\in C_3\cap C_1$, and then travels inside $C_1$ from $v_{31}$ back to $v_{12}$. We show there exists at least one such loop, hence $d<\infty$. Pick vertices $w_{ij}\in C_i\cap C_j$ and  $w\in \cap_{i=1}^3C'_i$. Let $p_{ij}$ be a path from $w$ to $w_{ij}$ inside $C'_i\cap C'_j$. Let $p_1$ be the $r_1$-image of the concatenation 
$p_{12}p_{31}$. Define $p_2,p_3$ similarly. The concatenation $p_1p_2p_3$ gives a loop satisfying all the conditions. Assume by contradiction that $\cap_{i=1}^3C_i=\emptyset$, then $d>0$. Again we pick a disk diagram $D$ with minimal number of squares among disk diagrams bounded by all such loops with length $=d$, and argue that there canot be spurs and length 2 subpaths in $\partial D$ forming the corner of a 2-cube in $D$ by using the local-convexity of $C_i$'s and their intersections. Thus $D$ is a point and we reach a contradiction as before.
\end{proof}

Let $Y\to Z$ be a map and let $\widehat Z\to Z$ be a covering map. An \emph{elevation} $\widehat Y\to \widehat Z$ of $Y\to Z$ is a map where the composition $\widehat Y\to Y\to Z$ equals $\widehat Y\to\widehat Z\to Z$, and such that choosing basepoints so the above maps are basepoint preserving, we have $\pi_1\widehat Y$ equals the preimage of $\pi_1 \widehat Z$ in $\pi_1 Y$.

In the special case when the graph of cube complexes $X$ has only one vertex space, Lemma~\ref{lem:intersection of cones} implies the following statement, which is of independent interest. 
\begin{cor}
	\label{cor:connected intersection}
Let $X$ be a compact special cube complex such that $\pi_1 X$ is word-hyperbolic. Let $\{f_i:Y_i\to X\}_{i=1}^n$ be a collection of local-isometries with compact domains. Then there exists a finite regular cover $\widehat X$ of $X$ such that for any $\{\widehat Y_j\to \widehat X\}_{j=1}^3$ where each $\widehat Y_j$ is an elevation of an element in $\{Y_i\}$, the following holds:
\begin{enumerate}
	\item Each $\widehat Y_j$ embeds.
	\item If $\widehat Y_1\cap \widehat Y_2\neq\emptyset$ then $\widehat Y_1\cap \widehat Y_2$ is connected.
	\item If $\{\widehat Y_1,\widehat Y_2,\widehat Y_3\}$ pairwise intersect then $\widehat Y_1\cap\widehat Y_2\cap\widehat Y_3\neq\emptyset$.
\end{enumerate}
\end{cor}

\begin{lem}
	\label{lem:intermediate special cover}
$\widehat X$ is a special cube complex.
\end{lem}

\begin{proof}
Recall that $\widehat X$ is a tree of cube complexes over $\bar{T}$. By Proposition~\ref{prop:graded}, each vertex space (hence each edge space) of $\widehat X$ is special. If $\widehat X$ fails to be special, then there exists a convex subcomplex of $\widehat X$ over a finite subtree of $\bar{T}$ that fails to be special. Thus it suffices to verify the condition of Theorem~\ref{thm:graph of cube complexes}

Theorem~\ref{thm:graph of cube complexes}.(1) follows from Lemma~\ref{lem:intersection of cones}.(1). Let $\widehat X_v$ be a vertex space of $\widehat X$ and let $\widehat Y\to\widehat X_v$ be an edge space of $\widehat X_v$ (we will treat $\widehat Y$ as a subspace of $\widehat X_v$). By Remark~\ref{lem:intersection of cones1}, $h\cap \widehat Y$ is connected for any hyperplane $h$ of $\widehat X_v$. Thus $h\cap \widehat Y$ is a hyperplane of $\widehat Y$. Thus Theorem~\ref{thm:graph of cube complexes}.(2) follows. For (3), choose two hyperplanes $h_1,h_2$ of $\widehat X_v$ such that $h_1\cap h_2\neq\emptyset$ and  both intersect $\widehat Y$, by Remark~\ref{lem:intersection of cones1}, $h_1\cap h_2\cap \widehat Y\neq\emptyset$, thus $h_1$ and $h_2$ cross inside $\widehat Y$. For (4), let $N_h$ be the carrier of a hyperplane $h$ of $\widehat X_v$ and suppose $h\cap \widehat Y\neq\emptyset$. Let $e_1\subset \widehat Y$ be an edge dual to $h$. Let $e$ be any edge dual to $h$. It suffices to show if $e\cap \widehat Y\neq\emptyset$, then $e\subset \widehat Y$. Since $N_h\cap \widehat Y$ is connected by Lemma~\ref{lem:intersection of cones}.(2), there is a combinatorial path in $N_h\cap \widehat Y$ from a point in $e_1$ to a point in $e$. Since $N_h$ is embedded (Lemma~\ref{lem:intersection of cones}.(1)), the local-convexity of $N_h\cap \widehat Y$
yields  $e\subset N_h\cap \widehat Y$.
\end{proof}

\begin{lem}
	\label{lem:special action}
$\bar{G}$ has a finite index subgroup that acts specially on $\widehat X$.
\end{lem}

\begin{proof}
Since $\bar{G}$ is virtually free, quasiconvex subgroups of $\bar{G}$ and double cosets of quasiconvex subgroups of $\bar{G}$ are separable by Lemma~\ref{lem:virtually free}.\eqref{vf2} below. Now we verify the assumption of Theorem~\ref{thm:shorter profinite criterion}. Note that $\bar{G}$ acts properly and cocompactly on $\widehat X$ by deck transformations, thus (1) and (2) follow. Since the collection of hyperplanes in $\widehat X$ is locally finite, the stabilizer of each hyperplane acts properly and cocompactly on it. Thus the stabilizer is finite generated, and is hence quasiconvex in $\bar{G}$ by Lemma~\ref{lem:virtually free}.\eqref{vf1}. Then Theorem~\ref{thm:shorter profinite criterion}.(4) follows. It remains to verify (3). Given hyperplanes $h_1,h_2$, we claim $H=\intersector_{\bar{G}}(h_1,h_2)$ is the union of finitely many double cosets of hyperplane stabilizers, and thus separable. Since $H$ is left $\stab(h_1)$-invariant and right $\stab(h_2)$-invariant, it is a union of double cosets of form $\stab(h_1)g\stab(h_2)$. Since $\stab(h_1)$ acts cocompactly on $h_1$, there are finitely many $\stab(h_1)$-orbits of hyperplanes of $\widehat X$ which intersect $h_1$. In particular, there exists a finite collection $\{g_i\}_{i=1}^{m}\subset H$ such that for any $g\in H$, there exists $g_i$ such that $gh_2$ and $g_ih_2$ are in the same $\stab(h_1)$-orbit. It follows that $H=\cup_{i=1}^{m}\stab(h_1)g_i\stab(h_2)$.
\end{proof}

\begin{lem}
	\label{lem:virtually free}
	Let $H$ be a f.g.\ free group. Then
	\begin{enumerate}
		\item \label{vf1} each f.g.\ subgroup of $H$ is quasiconvex;
		\item \label{vf2} each f.g.\ subgroup of $H$ is separable in $H$. Double cosets of f.g.\ quasiconvex subgroups are separable in $H$.
	\end{enumerate}
	The same conclusion holds if $H$ has a finite index subgroup that is a f.g.\ free group.
\end{lem}

\begin{proof}
Assertion (1) holds since subtrees of trees are convex.
The first part of (2) is proven by \cite{Hall49} and second part is proven in \cite{RibesZalesskii93}, with generalizations to separability of double cosets in hyperbolic groups \cite{Minasyan2006,GitikDoubleHyperbolicLERF}.
\end{proof}
	
\begin{proof}[Proof of Theorem~\ref{thm:main2}]
Let $\bar G'\le \bar G$ be a finite index subgroup acting specially on $\widehat X$, and let $G'\le G$ be the preimage of $\bar{G}'$ under $G\to\bar{G}$. Then $X'=X/G'=\widehat X/\bar G'$ is a finite sheet cover of $X$ that is special.
\end{proof}	

\subsection{Specialness implies finite stature}
\label{subsec:FGH}
Let $C$ be a convex subcomplex in a CAT$(0)$ cube complex $X$. Then there is a well-defined nearest point projection $\proj_{C}:X^0\to C^0$ with respect to the combinatorial distance on $X^0$, moreover, this projection extends to a cubical map $\proj_{C}:X\to C$, see \cite[Lem~13.8 and Rmk~13.9]{HaglundWiseSpecial}.
\begin{lem}
	\label{lem:intersection}
	Let $X$ be a compact nonpositively curved cube complex. Let $A$ and $B$ be  locally-convex subcomplexes of $X$. Suppose that for each immersed hyperplane $U\to X$, the map $N(U)\cong U\times[-\frac{1}{2},\frac{1}{2}]\to X$ is an embedding. Then for any pair of elevations $\widetilde A$ and $\widetilde B$ in $\widetilde X$, there exists a locally-convex subcomplex $K\subset A$ such that $K$ has an elevation $\widetilde K\subset\widetilde A$ with $\stab(\widetilde K)=\stab(\widetilde A)\cap\stab(\widetilde B)$. Moreover, $\widetilde K=\proj_{\widetilde A}\widetilde B$.
\end{lem}

\begin{proof}
As $A$ is embedded in $X$, $\stab(\widetilde A)=\pi_1 A$ up to conjugacy. First consider the case $\widetilde A\cap \widetilde B\neq\emptyset$. Let $\widetilde K=\widetilde A\cap \widetilde B$. Then $\widetilde K$ covers a connected component $K$ of $A\cap B$ and one readily verifies that such $K$ and $\widetilde K$ are as desired.

We prove the lemma by induction on the number of hyperplanes separating $\widetilde A$ and $\widetilde B$. Let $\tilde h$ be a hyperplane separating $\widetilde A$ and $\widetilde B$ such that the carrier $\widetilde N$ of $\tilde h$ satisfies $\widetilde N\cap \widetilde A\neq\emptyset$ (when $\widetilde A\cap\widetilde B\neq\emptyset$, such $\tilde h$ always exists, as any two disjoint convex subcomplexes of a CAT$(0)$ cube complex is separated by a hyperplane). Let $N$ be the embedded carrier in $X$ covered by $\widetilde N$. By the previous paragraph, $\widetilde K_1=\widetilde A\cap \widetilde N$ covers a component $K_1$ of $N\cap A$ and $\stab(\widetilde K_1)=\stab(\widetilde A)\cap \stab(\widetilde N)$. We identify $\widetilde N$ with $\tilde h\times[0,1]$ and identify $\widetilde K_1$ with $\widetilde K_1\times\{0\}$. Define $\widetilde K'_1=\widetilde K_1\times\{1\}$. As $N$ is embedded, $\stab(\widetilde K'_1)=\stab(\widetilde K_1)$ and $\widetilde K'_1$ covers a locally-convex subcomplex $K'_1\subset X$.
Now we claim:
\begin{enumerate}
	\item if $\tilde h'$ is a hyperplane separating $K'_1$ and $B$, then $\tilde h\cap \widetilde A=\emptyset$, hence $\tilde h'$ separates $\widetilde B$ from $\widetilde A$;
	\item $\stab(\widetilde A)\cap\stab(\widetilde B)\subset\stab(\widetilde N)$.
\end{enumerate}
For (1), we assume by contradiction that $\tilde h\cap \widetilde A\neq\emptyset$. Then $\tilde h'\cap \tilde h\neq\emptyset$, otherwise $\tilde h'$ and $\widetilde A$ are on the same side of $\tilde h$ and $\tilde h'$ cannot separate $K'_1$ and $B$. Then $\tilde h'$, $\widetilde N$ and $\widetilde A$ pairwise intersect, hence they have non-trivial intersection (cf.\ \cite[Lem~13.13]{HaglundWiseSpecial}). Hence $\tilde h'\cap  \widetilde K_1\neq\emptyset$ and $\tilde h'\cap \widetilde K_1\neq\emptyset$, which leads to a contradiction. For (2), let $g\in \stab(\widetilde A)\cap\stab(\widetilde B)$, then $g\tilde h$ separates $\widetilde A$ from $\widetilde B$ and $g\widetilde N\cap\widetilde A\neq\emptyset$. If $g\tilde h\cap \tilde h=\emptyset$, then $A$ and $g\tilde h$ are in the same side of $\tilde h$, and $A$ and $\tilde h$ are in the same side of $g\tilde h$. This contradicts that both $g\tilde h$ and $\tilde h$ separate $A$ from $B$. Thus $g\tilde h\cap \tilde h\neq\emptyset$. Hence $g\tilde h=\tilde h$ as each hyperplane of $X$ is embedded. Then (2) follows.

Claim (1) implies that the number of hyperplanes separating $K'_1$ and $\widetilde B$ is less than the number of hyperplanes separating $\widetilde A$ and $\widetilde B$. Thus by induction there is a locally-convex subcomplex $K'\subset K'_1$ with an elevation $\widetilde K'$ inside $\widetilde K'_1$ such that $\stab(\widetilde K')=\stab(\widetilde K'_1)\cap\stab(\widetilde B)=\stab(\widetilde K_1)\cap\stab(\widetilde B)=\stab(\widetilde A)\cap \stab(\widetilde N)\cap\stab(\widetilde B)=\stab(\widetilde A)\cap\stab(\widetilde B)$, where the last equality follows from claim (2). As $N$ is embedded, we can slide $K'$ through $N$ to obtain a locally-convex subcomplex $K\subset A$ as desired. The moreover statement follows from the construction of $\widetilde K$.
\end{proof}

\begin{thm}
	\label{thm:FGH}
Let $X$ be a compact nonpositively curved cube complex. Let $\{f_i:A_i\to X\}_{i=1}^n$ be a collection of local-isometries with compact domain. Suppose $X$ has a finite cover $\widehat X$ such that
\begin{enumerate}
	\item for each immersed hyperplane $U\to\widehat X$, the map $N(U)\cong U\times[-\frac{1}{2},\frac{1}{2}]\to \widehat X$ is an embedding;
	\item each elevation of $A_i$ to $\widehat X$ is an embedding.
\end{enumerate}
Then $(\pi_1 X,\{\pi_1 A_i\}_{i=1}^n)$ has finite stature.
\end{thm}

\begin{proof}
Without loss of generality, we assume $\widehat X\to X$ is regular. For each $A_i$, there are finitely many $\pi_1\widehat X$-conjugacy classes of subgroups of form $H\cap (\pi_1 A_i)^g$ for $g\in\pi_1 X$. Let $\{H_{ij}\}$ be their representatives. Let $\mathcal{C}_1$ (resp. $\mathcal{C}_2$) be the collection of intersection of conjugates (resp. $\pi_1 \widehat X$-conjugates) of elements in $\{\pi_1 A_i\}$ (resp. $\{H_{ij}\}$). Then each element of $\mathcal{C}_1$ contains an element of $\mathcal{C}_2$ as a subgroup of finite index that is  uniformly bounded  above. 

Let $\{\widehat B_{ij}\}$ be elevations of $\{A_i\}$ corresponding to $\{H_{ij}\}$. Then it follows from Lemma~\ref{lem:intersection} each element in $\mathcal{C}_2$ corresponds to a locally-convex subcomplex of $\widehat B_{ij}$. Thus $(\pi_1 \widehat X,\{H_{ij}\})$ has finite stature. To show $(\pi_1 X,\{\pi_1 A_i\}_{i=1}^n)$ has finite stature, we need to control elements in $\mathcal{C}_1$ that contain a given element in $\mathcal{C}_2$ as a finite index subgroup. Let $C_1$ and $C_2$ be two elevations of elements in $\{f_i:A_i\to X\}_{i=1}^n$ corresponding to two conjugates of elements in $\pi_1 A_i$. Then the intersection of these two conjugates stabilizes $\proj_{C_1}C_2$ and $\proj_{C_2}C_1$. By iterating this observation and using the moreover statement of Lemma~\ref{lem:intersection}, for each  $H \in \mathcal{C}_1$ there exists a convex subcomplex $\widetilde  K \subset \widetilde X$ such that $\stab_{\pi_1\widehat X}(\widetilde K)\le H\le \stab_{\pi_1X}(\widetilde K)$ and $\stab_{\pi_1\widehat X}(\widetilde K)\backslash \widetilde K $ is a subcomplex of $\widehat X$ corresponding to an element in $\mathcal{C}_2$. Thus the theorem follows.
\end{proof}

\begin{remark}
If stabilizers of hyperplanes in $\widetilde X$ are separable and each $\pi_1 A_i$ is separable, then the assumptions of Theorem~\ref{thm:FGH} are satisfied.
\end{remark}

\begin{cor}\
	\label{cor:FGH}
\begin{enumerate}
	\item Suppose $X$ is a compact nonpositively curved cube complex with a finite special cover. Let $\{f_i:A_i\to X\}_{i=1}^n$ be a collection of local-isometries with compact domain. Then  $(\pi_1 X,\{\pi_1 A_i\}_{i=1}^n)$ has finite stature.
	\item If $X$ is a compact graph of nonpositively curved cube complexes such that each vertex group is special, and each morphism from an edge space to a vertex space is an embedding. Then $\pi_1 X$ has finite stature with respect to its vertex groups.
    \item If $X$ is a compact graph of nonpositively curved cube complexes such that each vertex group is special, and each edge group is separable. Then $\pi_1 X$ has finite stature with respect to its vertex groups.
\end{enumerate}
\end{cor}

\begin{proof}
(1) holds by Theorem~\ref{thm:FGH} and the separability of $\pi_1 A_i$ (cf.\ \cite[Cor~7.9]{HaglundWiseSpecial} and \cite[Lem~8.1]{HaglundWiseSpecial}). (2) holds since by Lemma~\ref{lem:intersection} each transection corresponds to a locally-convex subcomplex of a vertex space. (3) follows from (2) as separability implies that $X$ has a finite cover satisfying the assumptions of (2).
\end{proof}

\begin{proof}[Proof of Theorem~\ref{thm:main}]
$(1)\Rightarrow(2)$ is Theorem~\ref{thm:main2}, and $(2)\Rightarrow(1)$ follows from Corollary~\ref{cor:FGH}.(1) and Lemma~\ref{lem:finitely many conjugacy classes0}.
\end{proof}

The following is a consequence of Theorem~\ref{thm:FGH} and Theorem~\ref{thm:main2}.
\begin{cor}
	\label{cor:separable hyperplanes}
Let $X$ split as a finite graph of compact nonpositively curved cube complexes such that each vertex group is word-hyperbolic. If the stabilizer of each hyperplane in $\widetilde X$ is separable. Then $X$ is virtually special.
\end{cor}

\begin{remark}
It is natural to ask whether Corollary~\ref{cor:connected intersection} holds under the weaker assumption that  edge groups are separable. By Theorem~\ref{thm:main2}, we need to show that separability of edges groups implies finite stature. This holds in the special case where the each edge space is superconvex in its vertex spaces (e.g. a graph of graphs), as we can use separability of edge groups to pass to a finite cover such that each edge space is superconvex and embedded in the vertex space, and then deduce from Lemma~\ref{lem:superconvex fiber-product intersection} that each transection can be realized as the fundamental group of a locally-convex subcomplex of some vertex space. However, the general case is not clear.
\end{remark}
	
\bibliographystyle{alpha}
\bibliography{wise}

%
%
\end{document}